\documentclass[10pt]{amsart}

\usepackage{enumitem}
\usepackage{xcolor,graphicx}
\usepackage{hyperref}
\usepackage{tikz}
\usepackage{amsmath, amsfonts, latexsym, array, amssymb, amscd }
\usetikzlibrary{shapes,arrows,positioning,decorations.pathreplacing}

\usepackage{fancyhdr,datetime2}
\newcommand{\Ar}
{\mathcal{L}_A}
\newcommand{\ph}{\varphi}

\newcommand{\eps}{\varepsilon}

\newcommand{\RR}{\mathbb{R}}

\newcommand{\CCC}{\mathcal{C}}

\newcommand{\PPP}{\mathcal{P}}
\newcommand{\SSS}{\mathcal{S}}
\newcommand{\BBB}{\mathcal{B}}
\newcommand{\UUU}{\mathcal{U}}
\newcommand{\GGG}{\mathcal{G}}

\newcommand{\NN}{\mathbb{N}}

\usepackage{ulem}

\numberwithin{equation}{section}

\newtheorem{theorem}{Theorem}[section]
\newtheorem{proposition}[theorem]{Proposition}
\newtheorem{lemma}[theorem]{Lemma}
\newtheorem{definition}[theorem]{Definition}
\newtheorem{remark}[theorem]{Remark}
\newtheorem{corollary}[theorem]{Corollary}

\newtheorem{thma}{Theorem}

\begin{document}

\title[Thermodynamics with controlled specification]{Thermodynamic formalism for non-uniform systems with controlled specification and entropy expansiveness}

\author{Tianyu Wang}
\author{Weisheng Wu}

\address{School of Mathematical Sciences, Shanghai Jiaotong University, No 800 Dongchuan Road, Shanghai 200240, P.R.China}
\email{a356789xe@sjtu.edu.cn}

\address{School of Mathematical Sciences, Xiamen University, Xiamen, 361005, P.R. China}
\email{wuweisheng@xmu.edu.cn}

\begin{abstract}
We study thermodynamic formalism of dynamical systems with non-uniform structure. Precisely, we obtain the uniqueness of equilibrium states for a family of non-uniformly expansive flows by generalizing Climenhaga-Thompson's orbit decomposition criteria. In particular, such family includes entropy expansive flows. Meanwhile, the essential part of the decomposition is allowed to satisfy an even weaker version of specification, namely controlled specification, thus also extends the corresponding results in Pavlov \cite{Pav19}.

Two applications of our abstract theorems are explored. Firstly, we introduce a notion of regularity condition called weak Walters condition, and study the uniqueness of measure of maximal entropy for a suspension flow with roof function satisfying such condition. Secondly, we investigate topologically transitive frame flows on rank one manifolds of nonpositive curvature, which is a group extension of nonuniformly hyperbolic flows. Under a bunched curvature condition and running a Gauss-Bonnet type of argument, we show the uniqueness of equilibrium states with respect to certain potentials.
\end{abstract}



\maketitle

\section{Introduction}
Let $X$ be a compact metric space, $F=\{f_t\}_{t\in \mathbb{R}}$ be a continuous flow over $X$, and $\varphi\in C(X,\mathbb{R})$ be a given potential function. We will study the general conditions exerted on $(X,F,\varphi)$ that guarantee the existence and uniqueness of equilibrium states. One of the earliest general criteria traces back to the remarkable work of Bowen in 1970's \cite{Bow74}, which shows that in the case of homeomorphism $(X,f,\ph)$, expansiveness and specification of $(X,f)$ plus a bounded distortion property of $\ph$ (also known as the Bowen property nowadays) imply that $(X,f,\ph)$ has a unique equilibrium state. To a significant level, these three conditions are enlightened by the corresponding phenomena in uniformly hyperbolic systems. Later in 2016, by weakening the above uniform conditions in different directions, Climenhaga and Thompson managed to improve Bowen's criteria in \cite{CT16} by making it applicable to amounts of systems exhibiting various types of non-uniform hyperbolicity. Precisely, they relax the above three uniform conditions and decompose every orbit segment from $X\times \mathbb{R}^+$\footnote{where each pair $(x,t)\in X\times \mathbb{R}^+$ is thought of as the orbit segment $\{f_s(x)|0\leq s<t\}$.} into three parts $(\mathcal{P},\mathcal{G},\mathcal{S})$, with $\mathcal{G}$ satisfying weakened specification and Bowen property and dominating in pressure. The non-uniform expansiveness assumption exerted there is expressed by so-called `obstruction to expansiveness', which considers examples in which the non-expansive set can not support any measures with large pressure. In particular, the failure of expansiveness is `invisible' with respect to the equilibrium state.

Clearly, such requirement on expansiveness can not always be satisfied when the system is only assumed to be \emph{entropy expansive}, as the flow may act as an isometry in central direction with dimension greater than $1$, causing the non-expansive set (at any scale) to be the entire space. In this paper, motivated by Climenhaga-Thompson's orbit decomposition technique, we propose a general criteria in characterizing thermodynamic formalism for entropy-expansive systems. Meanwhile, the specification property to our concern is a weakened version of \emph{controlled specification}. The optimism of such property in providing uniqueness of equilibrium states is first revealed in \cite{Pav16} by showing that the failure of controlled specification may cause co-existence of multiple equilibrium states, even in symbolic dynamics. Being our first main result, the criteria is stated as follows.
\begin{thma} \label{thm:a}
    Let $(X,F)$ be a continuous flow on a compact metric space, and $\ph:X\to \mathbb{R}$ be a continuous potential. Suppose there exists some $\eps>0$ at which $(X,F)$ is entropy expansive, and there exists  $\mathcal{D}\subset X\times\mathbb{R}^+$ which admits a decomposition $(\mathcal{P},\mathcal{G},\mathcal{S})$ satisfying the following conditions:
    \begin{enumerate}
        \item For every $\eta>0$, $\mathcal{G}$ satisfies weak controlled specification at scale $\eta$ with gap function $h_{\eta}$.
        \item $\ph$ is $g$-distorted over $X\times \mathbb{R}^+$ at scale $\eps$.
        \item For every $\eta>0$, $h_{\eta}$ and $g$ satisfy $\liminf_{t\to \infty}\frac{h_{\eta}(t)+g(t)}{\log t}=0$.
        \item $\lim_{\eta\to 0}P(\mathcal{D}^c\cup [\mathcal{P}]\cup [\mathcal{S}],\ph,\eta,\eps)<P(\ph)$.
    \end{enumerate}
    Then $(X,F,\ph)$ has a unique equilibrium state.
\end{thma}

As an immediate corollary, by taking $\mathcal{D}=X\times \mathbb{R}^+$, $\mathcal{P}=\mathcal{S}=\emptyset$, and $\ph=0$, we have the following criteria in formulating the corresponding result for \emph{measure of maximal entropy}.
\begin{thma} \label{thm:mme}
    Let $(X,F)$ be a continuous flow on a compact metric space. Suppose there exists some $\eps>0$ at which $(X,F)$ is entropy expansive, and for every $\eta>0$, the flow satisfies weak controlled specification at scale $\eta$ with gap function $h_{\eta}$ satisfying $\liminf_{t\to \infty}\frac{h_{\eta}(t)}{\log t}=0$. Then $(X,F)$ has a unique measure of maximal entropy.
\end{thma}

Unlike \cite[Theorem 2.9]{CT16} and \cite{PYY22}, our specification property must be assumed at all scales. Otherwise, uniqueness of equilibrium states can easily fail if specification only holds at a fixed scale. For instance, one may think of a constant suspension over a transitive Anosov diffeomorphism on a compact connected metric space as an easy counterexample.

If we keep the old non-uniform expansiveness assumption from \cite{CT16}, the above criteria still holds, with an even weaker regularity assumption on $\ph$.
\begin{thma} \label{thm:a'}
    Let $(X,F)$ be a continuous flow on a compact metric space, and $\ph:X\to \mathbb{R}$ be a continuous potential. Suppose there exists some $\eps>0$ such that $P_{\exp}^{\perp}(\ph,\eps)<P(\ph)$, and there exists $\mathcal{D}\subset X\times \mathbb{R}^+$ which admits a decomposition $\mathcal{P},\mathcal{G},\mathcal{S}$ satisfying conditions (1), (3) and (4) as in Theorem \ref{thm:a}, as well as $\ph$ being $g$-distorted over just $\mathcal{G}$ at scale $\eps$. Then $(X,F,\ph)$ has a unique equilibrium state.
\end{thma}

Both Theorems \ref{thm:a} and \ref{thm:a'} serve as consequences of the following more general result, with details being clarified in $\mathsection 2.4$ and $\mathsection 3$.

\begin{thma} \label{thm:weprove}
     Let $(X,F)$ be a continuous flow on a compact metric space, and $\ph:X\to \mathbb{R}$ be a continuous potential. Suppose there exists some $\eps>0$ 
     such that
     $$
P(\ph,\eta)=P(\ph) \text{ for all }\eta\in (0,\eps/2)
     $$
     and for every $t>0$, any finite partition $\mathcal{A}_t$ with $\text{Diam}_t(\mathcal{A}_t)<\eps$, and any $\mu\in \mathcal{M}_F^e(X)$ being almost entropy expansive at scale $\eps$, we have
     $$
h_{\mu}(f_t,\mathcal{A}_t)=h_{\mu}(f_t),
     $$
     plus the existence of some $\mathcal{D}\subset X\times\mathbb{R}^+$ admitting a decomposition $(\mathcal{P},\mathcal{G},\mathcal{S})$ satisfying the following conditions:
    \begin{enumerate}
        \item For every $\eta>0$, $\mathcal{G}$ satisfies weak controlled specification at scale $\eta$ with gap function $h_{\eta}$.
        \item $\ph$ is $g$-distorted over $\mathcal{G}$ at scale $\eps$.
        \item For every $\eta>0$, $h_{\eta}$ and $g$ satisfy $\liminf_{t\to \infty}\frac{h_{\eta}(t)+g(t)}{\log t}=0$.
        \item $\lim_{\eta\to 0}P(\mathcal{D}^c\cup [\mathcal{P}]\cup [\mathcal{S}],\ph,\eta,\eps)<P(\ph)$.
    \end{enumerate}
    Then $(X,F,\ph)$ has a unique equilibrium state.
\end{thma}

We also highlight two examples to which the above results apply.
\begin{enumerate}
      \item Suspension flow $(X,f,r)$ with $f$ being an Lipschitz homeomorphism over compact $X$ and roof function $r$ satisfying a weak regularity condition, namely \emph{weak Walters condition}. Details will be explored in $\mathsection 4.1$.
    \begin{thma}  \label{thm:suspension}
    Let $(X,f)$ be a Lipschitz homeomorphism over compact metric space, and $r\in C(X,\mathbb{R}^+)$. Suppose that $(X,f,r)$ satisfies the following conditions:
    \begin{itemize}
        \item $(X,f)$ is expansive. 
        \item For every $\delta>0$, there exists a monotonically increasing $h=h_{\delta}:\mathbb{N}^+\to \mathbb{N}^+$ such that $\liminf\frac{h(n)}{\log n}=0$, and $(X,f)$ satisfies weak controlled specification at scale $\delta$ with gap function $h$. 
        \item $r$ satisfies weak Walters condition.
    \end{itemize}
    Then the associated suspension flow $(X_r,F)$ has a unique measure of maximal entropy.
    \end{thma}
    \item Topologically transitive frame flows on rank one manifolds of nonpositive curvature, satisfying a bunched curvature assumption.  Details will be explored in $\mathsection 4.2$.
    \begin{thma}\label{frameflow}
Let $M$ be a closed, oriented, nonpositively curved $n$-manifold, with bunched curvature. Suppose that the frame flow $F^t:FM\to FM$ is topologically transitive, and $\tilde\varphi: FM\to \mathbb{R}$ is a continuous potential that is constant on the fibers of the bundle $FM\to SM$, whose projection $\varphi :SM\to \RR$ is of the form $\varphi=q\phi_u$ or is H\"{o}lder continuous. Moreover, assume that $P(g^t, \text{Sing},\varphi) <P(g^t, \varphi)$. Then there is a unique equilibrium measure for $(F^t, \tilde\varphi)$. 
\end{thma}
\end{enumerate}

\section{Preliminaries}
\subsection{Partition sums and topological pressure}
 Let $F=(f_t)_{t\in \mathbb{R}}$ be a continuous flow on a compact metric space $(X,d)$, and $\mathcal{M}_F(X)$ (resp. $\mathcal{M}_F^e(X))$ denote the set of all invariant (resp. ergodic) probability measures on $X$ under $(f_t)_{t\in \mathbb{R}}$. Define a new metric
 $$d_t(x,y)=d_t^F(x,y):=\max_{0\le s\le t}d(f_sx, f_sy), \quad \forall x,y\in X.$$
 Then given $\gamma>0$ and time $t>0$, the $t$-th Bowen ball centered at $x\in X$ is defined as
 $$B_t(x,\gamma):=\{y\in X: d_t(x,y)<\gamma\}.$$
 
 For the purpose of being explicit, we inherit the usage of two-scale pressure from \cite{CT16}. Recall that given a potential $\ph: X\to \mathbb{R}$, a scale $\gamma$ and time $t>0$, the representative information of $\ph$ over a $t$-th Bowen ball centered at $x$ is given by
$$
\Phi_{\gamma}(x,t):=\sup_{y\in B_t(x,\gamma)}\int_0^t \ph(f_sy)ds.
$$
For convenience, we also write $\Phi_t(x)=\Phi(x,t):=\Phi_0(x,t)$.

Given $\mathcal{C}\subset X\times \mathbb{R}^+$ and $t,\delta,\gamma>0$, the (separated) partition function of $\ph$ is given by
$$
\Lambda_t^{F,d}(\mathcal{C},\ph,\delta,\gamma)=\Lambda_t(\mathcal{C},\ph,\delta,\gamma):=\sup\{\sum_{x\in E}e^{\Phi_{\gamma}(x,t)}|E\subset \mathcal{C}_t \text{ is } (t,\delta)\text{-separated}\}.
$$
Respectively, the pressure of $\ph$ over $\mathcal{C}$ at scale $(\delta,\gamma)$ is given by
$$
P(F,\mathcal{C},\ph,\delta,\gamma)=P(\mathcal{C},\ph,\delta,\gamma)=\limsup_{t\to \infty}\frac{1}{t}\log \Lambda_t(\mathcal{C},\ph,\delta,\gamma).
$$
The above definition differs from the classic definition of topological pressure mainly in the presence of parameter $\gamma$. When no distortion is considered, we simply have $\Lambda_t(\mathcal{C},\ph,\delta):=\Lambda_t(\mathcal{C},\ph,\delta, 0)$. Write also $$P(F,\mathcal{C},\varphi,\delta)=P(\mathcal{C},\varphi,\delta):=P(\mathcal{C},\varphi,\delta,0), \text{\ and\ } P(\varphi,\delta):=P(X\times \mathbb{R}^+,\varphi,\delta).$$
Then $P=P(\varphi):=\lim_{\delta\to 0}P(\varphi,\delta)$.

The adaption to a homeomorphism $f:X\to X$ of the definition of partition function for degree $n\in \mathbb{N}$ is straightforward, as we simply replace the integral in $\Phi_{\gamma}$ by the partition sum $S_n\varphi(y):=\sum_{0\leq j \leq n-1}\varphi(f^j(y))$.
\begin{remark} \label{remarkE}
    Throughout the paper, for each $E\subset X$, we will abbreviate $E\times \mathbb{R}^+$ as $E$ when no confusion is raised.
\end{remark}

Let $\mathcal{C},t,\delta,\gamma$ be as above. 
As an analogy to the separated partition function, we give the definition of spanning partition function as follows
$$
\Lambda_{sp,t}^{F,d}(\mathcal{C},\varphi,\delta,\gamma)=\Lambda_{sp,t}(\mathcal{C},\varphi,\delta,\gamma):=\inf\{\sum_{x\in E}e^{\Phi_{\gamma}(x,t)}|E\subset \mathcal{C}_t \text{ is }(t,\delta) \text{-spanning}\}
$$
with a straightforward adaption to homeomorphisms. Also write 
$$\Lambda_{sp,t}(\mathcal{C},\varphi,\delta)=\Lambda_{sp,t}(\mathcal{C},\varphi,\delta,0).$$ 
The following observation follows from a standard argument, which can be found in \cite[page 169]{Wal82}.
\begin{proposition}
    For every $\mathcal{C}\subset X\times \mathbb{R}^+$ and $t,\delta>0$, we have
\begin{equation} \label{eqsepspancomparison}
    \Lambda_t(\mathcal{C},\varphi,2\delta)e^{-\text{Var}_t(\varphi,\delta)}\leq \Lambda_{sp,t}(\mathcal{C},\varphi,\delta) \leq \Lambda_{t}(\mathcal{C},\varphi,\delta),
\end{equation}
where $\text{Var}_t(\varphi,\delta):=\sup_{x\in X, d_t(x,y)<\delta}|\Phi(x,t)-\Phi(y,t)|$.
\end{proposition}

\subsection{Variations and distortions}
Given any constant $\delta>0$, $T>0$ and $\mathcal{C}\subset X\times \mathbb{R}^+$, the ($T$-th) partial sum variation of $\ph$ at scale $\delta$ is given by
$$
\text{Var}_T(\mathcal{C},\delta):=\sup_{(x,T)\in \mathcal{C},d_T(x,y)<\delta}|\Phi(x,T)-\Phi(y,T)|.
$$
We also artificially let $\text{Var}_T(\mathcal{C},\delta):=-\infty$ when $\mathcal{C}_T=\emptyset$.
Let $g$ be a positive function defined on $\mathbb{R}^+$. We say $\ph$ is $g$-distorted over $\mathcal{C}$ and at scale $\delta$ if
$$
\text{Var}_T(\mathcal{C},\delta)<g(T).
$$

\subsection{Specification} The specification property we use throughout this paper is a weakened version of both weak specification in \cite{CT16} and controlled specification in \cite{Pav19}, namely \textit{weak controlled specification}.
\begin{definition}
Let $h$ be a non-decreasing function defined on $\mathbb{R}^+$, and $\delta>0$ be a constant. We say $\mathcal{C}\subset X\times \mathbb{R}^+$ satisfies \emph{weak controlled specification property} with gap function $h$ at scale $\delta$ if for every $n\geq 2$ and $\{(x_i,t_i)\}_{i=1}^n\subset \mathcal{C}$, there exist
\begin{enumerate}
    \item a sequence $\{\tau_i\}_{i=1}^{n-1}$ with $\tau_i\leq \max\{h(t_i),h(t_{i+1})\}$ for each $i$,
    \item a point $y\in X$, for which we denoted by $\text{Spec}^{n,\delta}_{\{\tau_i\}}(\{(x_i,t_i)\})$
\end{enumerate}
such that
$$
d_{t_i}(f_{\sum_{j=1}^{i-1}(t_j+\tau_j)}(y),x_i)<\delta.
$$
\end{definition}
We will neglect the dependence on $h$ by simply writing weak controlled specification at a specific scale when the gap function is prefixed.
\begin{remark}
    We will assume that $\min\{h(t),g(t)\}\geq 1$ for all $t\geq 0$ throughout this paper. Such an assumption causes no harm to the respective properties due to the nature of weak specification and distortion upper bound.
\end{remark}

\subsection{Decompositions}
As can be glanced in $\mathsection 1$, our main abstract results are stated in the language of \emph{orbit decomposition}, whose elegance is first discovered in \cite{CT12}, and then in a much greater generality in \cite{CT16}. We briefly recall the definition as follows.
\begin{definition} \label{def:decomp}
    A decomposition $(\mathcal{P},\mathcal{G},\mathcal{S})$ for $\mathcal{D}\subset X\times \mathbb{R}^+$ consists of three collections $\mathcal{P},\mathcal{G},\mathcal{S} \subset X\times \mathbb{R}^+$ and functions $p,g,s:\mathcal{D}\to \mathbb{R}^+$ such that for every $(x,t)\in \mathcal{D}$, writing $p(x,t), g(x,t), s(x,t)$ as $p,g,s$ respectively, we have $t=p+g+s$, and
    $$
(x,p)\in \mathcal{P}, \qquad (f_p(x),g)\in \mathcal{G}, \qquad (f_{p+g}(x),s)\in \mathcal{S} 
    $$
    Given a decomposition $(\mathcal{P},\mathcal{G},\mathcal{S})$ and constant $c\geq 0$, we denote by $\mathcal{G}^c$ the collection of $(x,t)\in \mathcal{D}$ satisfying $\max\{p(x,t),s(x,t)\}\leq c$. Meanwhile, $X\times \{0\}$ is always assumed to belong to all three collections $(\mathcal{P},\mathcal{G},\mathcal{S}$ by default.
\end{definition}

\begin{remark} \label{remark:G1G}
    Given a decomposition $(\mathcal{P},\mathcal{G},\mathcal{S})$ for $\mathcal{D}$, we will have to frequently deal with $\mathcal{G}^1$ due to a discretization process. Meanwhile, it follows immediately from the definition of $\mathcal{G}^1$ that both the distortion control and specification on $\mathcal{G}$ can be inherited to $\mathcal{G}^1$. Precisely, we have the following
    \begin{itemize}
        \item If $\varphi$ is $g$-distorted over $\mathcal{G}$ at scale $\delta$, then it is $(g+4||\varphi||)$-distorted over $\mathcal{G}^1$ at the same scale, where $||\varphi||:=\max\{|\varphi(x)|:x\in X\}$. 
        \item If $\mathcal{G}$ satisfies weak controlled specification property with gap function $h_{\mathcal{G}}$ at scale $\delta$, then it is so for $\mathcal{G}^1$ with the same gap function at scale $\delta'$, where $\delta':=\max\{d(f_s(x),f_s(y)):d(x,y)\leq \delta, s\in [0,1]\}$.  
    \end{itemize}
    In particular, $\mathcal{G}$ satisfying weak controlled specification at all scales is equivalent to which for $\mathcal{G}^1$.
\end{remark}

\subsection{Expansiveness}
For every $x\in X$ and $\eps>0$, consider the corresponding two-sided infinite Bowen ball given by
$$
\Gamma_{\eps}(x):=\{y\in X:d(f_t(x),f_t(y))\leq\eps \text{ for all }t\in \mathbb{R}\}.
$$
The size of $\{\Gamma_{\eps}(x)\}_{x\in X}$ is frequently used in characterizing non-expansiveness of $(X,F)$. Precisely, we consider two versions of weakened expansiveness in this paper. The first is known as ``entropy expansiveness'', which concerns the topological entropy of $\Gamma_{\eps}(x)$ and is defined as follows.
\begin{definition}
    Given $\eps>0$, we say $(X,F)$ is \emph{entropy expansive at scale $\eps$} if for every $x\in X$, we have $h(\Gamma_{\eps}(x))=0$. $(X,F)$ is \emph{entropy expansive} if it is so at some scale. 
\end{definition}

The second is defined using so-called obstruction to expansiveness, which is a measure-theoretic analog of expansiveness in asking for almost expansiveness for all measures with large entropy. Given $\eps>0$, let the set of \emph{non-expansive points at scale $\eps$} for $F$ be 
$$
\text{NE}(\eps):=\{x\in X:\Gamma_{\eps}(x)\nsubseteq f_{[-s,s]}(x)\text{ for any }s>0\}.
$$
The following quantity introduced in \cite[$\mathsection 2.5$]{CT16} captures the largest information of a non-expansive ergodic measure.
\begin{definition}
    Given a potential function $\ph\in C(X, \mathbb{R})$ and $\eps>0$, the \emph{pressure of obstructions to expansiveness at scale $\eps$} is defined as
    $$
P^{\perp}_{\exp}(\ph,\eps):=\sup_{\mu\in \mathcal{M}_F^e(X)}\Bigl \{h_{\mu}(f_1)+\int \ph d\mu:\mu(\text{NE}(\eps))=1  \Bigl \},
    $$
    and
    $$
P^{\perp}_{\exp}(\ph)=\lim_{\eps\to 0}P^{\perp}_{\exp}(\ph,\eps).
    $$
\end{definition}

Notice that if $\mu\in \mathcal{M}_F^e(X)$ satisfies $h_{\mu}(f_1)+\int \ph d\mu>P^{\perp}_{\exp}(\ph,\eps)$, then $\mu(\text{NE}(\eps))=0$. In other words, $\mu$ is \emph{almost expansive} at scale $\eps$. In particular, it follows from the continuity of the flow and compactness of $X$ that $\mu$ is \emph{almost entropy expansive}, meaning that $h(\Gamma_{\eps}(x))=0$ for $\mu$-a.e. $x$. It is also clear that when $(X,F)$ is entropy expansive, every $\mu\in \mathcal{M}_F(X)$ is almost entropy expansive. 

One of the essential features for systems with appropriate level of expansiveness is that both topological and measure-theoretic information are reflected at a given scale, on which we are able to run our estimates. Details are shown as follows.
\begin{proposition} \label{propexpansivegood}
    Let $(X,F)$ be a continuous flow over a compact metric space, $\ph\in  C(X, \mathbb{R})$, and $\eps>0$ a constant.
    \begin{enumerate}
        \item 
        If $P^{\perp}_{\exp}(\ph,\eps)<P(\ph)$, then
         \begin{equation} \label{eqexpansivescale}
        P(\ph,\eta)=P(\ph) \text{ for all } \eta\in (0,\eps/2). 
         \end{equation}
        If $(X,F)$ is entropy expansive at scale $\eps$, and $\ph$ is $g$-distorted over $X\times \mathbb{R}^+$ at scale $\eps$ with $g$ satisfying $\lim_{T\to \infty}\frac{g(T)}{T}=0$, then \eqref{eqexpansivescale} also holds true.
        \item For every $t>0$, let $\mathcal{A}_t$ be any finite measurable partition of $X$ with $\text{Diam}_t(\mathcal{A}_t)<\eps$, where $\text{Diam}_t$ refers to the maximal diameter of elements in $\mathcal{A}_t$ regarding the metric $d_t$. Then for every $\mu\in \mathcal{M}_F^e(X)$ being almost entropy expansive at scale $\eps$, we have
         $$h_{\mu}(f_t,\mathcal{A}_t)=h_{\mu}(f_t).$$          
         As discussed above, this covers all elements in $\mathcal{M}_F^e(X)$ if $(X,F)$ is entropy expansive; and covers every ergodic $\mu\in \mathcal{M}_F^e(X)$ satisfying $h_{\mu}(f_1)+\int \ph d\mu>P^{\perp}_{\exp}(\ph,\eps)$.
    \end{enumerate}
\end{proposition}
\begin{proof}
    We only prove \eqref{eqexpansivescale} in the case of entropy expansive $(X,F)$ as the others are just restatements of \cite[Theorem 3.2 $\&$ Proposition 3.7]{CT16}. We first consider $(X,f_t,d_t)$ for any fixed $t>0$. For every $E\subset X$ and $\delta>0$, denote by $r_n^{f_t,d_t}(E,\delta)$ the cardinality of a minimal $(\delta,n)$-spanning set for $E$ under $f_t$ within metric $d_t$. Fix $\eta\in (0,\eps/2)$, and write $\eta':=\eta/2$. Since a bi-infinite Bowen ball for $f_t$ within metric $d_t$ coincides with a bi-infinite Bowen ball for flow, it follows from \cite[Proposition 2.2] {Bow72} that for every $\beta>0$, there exists a constant $C_{\beta}=C_{\beta}(\eta'/2)>0$ such that
    $$
r_n^{f_t,d_t}(B_n^{d_t}(x,\eps),\eta'/2)\leq C_{\beta}e^{n(h(f_t)+\beta)}.
    $$
    Following the proof of \cite[Theorem 2.4] {Bow72}, we have
    $$
\Lambda_{sp,n}^{f_t,d_t}(X,\ph,\eta'/2)\leq C_{\beta}e^{g(tn)+n\beta}\Lambda_{sp,n}^{f_t,d_t}(X,\ph,\eps), 
    $$
    which together with \eqref{eqsepspancomparison} implies that
    $$
    e^{-g(tn)}\Lambda_n^{f_t,d_t}(X,\Phi_t,\eta') 
    \leq \Lambda_{sp,n}^{f_t,d_t}(X,\Phi_t,\eta'/2) \leq C_{\beta}e^{g(tn)+n\beta}\Lambda_n^{f_t,d_t}(X,\Phi_t,\eps).
    $$
    Therefore, we have
    \begin{equation} \label{eqestimateeta'}
       \Lambda_n^{f_t,d_t}(X,\Phi_t,\eta')
       \leq C_{\beta}e^{2g(tn)+n\beta}\Lambda_n^{f_t,d_t}(X,\Phi_t,\eps)=C_{\beta}e^{2g(tn)+n\beta}\Lambda_{nt}(X,\varphi,\eps).
    \end{equation}
    
    For every $m\in \mathbb{N}$, let $E_m$ be a maximally $(\eta',m)$-separated set for $\ph$ under $f_t$ within metric $d_t$. For every $T>0$, write $T=nt+t'$ with $n\in \mathbb{N}$ and $t'\in (0,t]$. Let $E_T$ (resp. $E_{t'}$) be a maximally $(\eta,T)$-separated set (resp. $(\eta',t')$-separated set) for $\ph$ under the original flow $F$. Define a map $\pi:E_{T}\to E_{n}\times E_{t'}$ as follows: for $x\in E_T$, $\pi(x)=(x_1,x_2)$ is such that $d_{nt}(x,x_1)<\eta'$ and $d_{t'}(f_{nt}(x),x_2)<\eta'$. Such a map $\pi$ is injective by the definition, which implies that
    $$
\Lambda_T(X,\varphi,\eta)\leq e^{g(tn)+g(t')}\Lambda_n^{f_t,d_t}(X,\Phi_t,\eta')\Lambda_{t'}(X,\varphi,\eta').
    $$
    By noticing that $\Lambda_{t'}(X,\varphi,\eta')$ is bounded from above by a constant $M(t,\eta')$ and invoking \eqref{eqestimateeta'}, we have
    $$
\Lambda_T(X,\varphi,\eta)\leq C_{\beta}M(t,\eta')e^{3g(tn)+g(t')+n\beta}\Lambda_{nt}(X,\varphi,\eps).
    $$
    Since $\lim_{T\to \infty}\frac{g(T)
    }{T}=0$ and $\sup_{t'\in [0,t]}g(t')$ is bounded, by taking logarithm over above both sides of the inequality, dividing by $nt$ and making $n$ go to infinity, we have
    $$
P(\varphi,\eta)\leq P(\varphi,\eps)+\frac{\beta}{t}.
    $$
    Since the choice of $\beta$ is arbitrary, we conclude the proof of the proposition by making $\beta$ go to $0$.
    \end{proof}

    When $P(\ph,\eta)=P(\ph)$, it follows from the proof of \cite[Lemma 4.14]{CT16} that $(X,F,\ph)$ has an equilibrium state. Therefore, Proposition \ref{propexpansivegood}(1) provides us with the existence of equilibrium state in our setting.
    \begin{proposition} \label{propexistence}
        In both of the following two cases, $(X,F,\ph)$ has an equilibrium state.
        \begin{enumerate}
            \item $(X,F)$ is entropy expansive at scale $\eps>0$, at which $\ph$ is $g$-distorted over $X\times \mathbb{R}^+$ with $g$ satisfying $\lim_{T\to \infty}\frac{g(T)}{T}=0$.
            \item $P^{\perp}_{\exp}(\ph)<P(\ph)$.
        \end{enumerate}
    \end{proposition}
    It is now clear from Propositions \ref{propexpansivegood} and \ref{propexistence} that both Theorem \ref{thm:a} and Theorem \ref{thm:a'} follow from Theorem \ref{thm:weprove}, whose proof will be given in $\mathsection 3$.

\subsection{Maximal ergodic theorem and maximal inequality}
Before stepping into the proof of Theorem \ref{thm:weprove}, we give the following version of maximal ergodic theorem in the flow case, which is essential in deducing uniqueness of equilibrium states in our setting.
\begin{lemma}
    Let $\mu\in \mathcal{M}_F(X)$. For $g\in L^1(\mu)$ and $x\in X$, write 
    $$M^+g(x):=\sup_{t>0}\left\{\frac{1}{t}\int_0^t g(f_s(x))ds\right\}.$$ Then for every $\lambda\in \mathbb{R}$ we have
    $$
\int_{M^+g>\lambda}(g-\lambda) d\mu \geq 0.
    $$
    Consequently, when $\lambda>0$, we have
    \begin{equation} \label{eqM+g}
        \mu(M^+g\leq \lambda)\geq 1-\frac{||g||_{L^1(\mu)}}{\lambda}.
    \end{equation}
    Similarly, the above inequality still holds true for 
    $$M^-g(x):=\sup_{t>0}\left\{\frac{1}{t}\int_{-t}^0 g(f_{s}(x))ds\right\}$$ if we consider flowing backwards instead, which is
        \begin{equation} \label{eqM-g}
        \mu(M^-g\leq \lambda)\geq 1-\frac{||g||_{L^1(\mu)}}{\lambda}.
    \end{equation}
\end{lemma}

\section{Proof of Theorem \ref{thm:weprove}}
\subsection{Upper bound of partition sums}
This section is devoted to the proof of Theorem \ref{thm:weprove}. Throughout the entire section, $\eps$ will be the fixed scale at which obstructions to expansiveness and regularity are well-controlled, as mentioned in the statement of Theorem \ref{thm:a}. The other constant $\eta$ representing (a multiple of) specification scale will always be assumed to be small compared to $\eps$,  i.e. $\eta\in (0,\eps/40)$. As briefly mentioned in Remark \ref{remark:G1G}, our proof relies heavily on a discretization process. Consequently, throughout this section, we majorly deal with $\mathcal{G}^1$, over which specification property and distortion control are handled. Based on Remark \ref{remark:G1G}, by expanding the size of $g$ by a fixed constant if necessary, we may assume that $\mathcal{G}^1$ is $g$-distorted at scale $\eps$. We also denote by $h_{\eta}^{\mathcal{G}^1}$ the weak specification gap function for $\mathcal{G}^1$ at scale $\eta$, and write $h_{\eta}^{\mathcal{G}^1}$ as $h$ when $\eta$ is clear.
\begin{lemma} \label{lemmaGupperbound}
    For every small $\eta>0$, there exists a constant $C_1=C_1(\mathcal{G}^1,5\eta)>2$ such that for any $n\in \mathbb{N}$ and $t>0$, we have
\begin{equation} \label{equpperboundmultipleG}
    (\Lambda(\mathcal{G}^1,5\eta,t))^n
\leq (e^{C_1h(t)+g(t)})^n\Lambda(X,\eta,nt+(n-1)h(t)).
\end{equation}
\end{lemma}
\begin{proof}
    The proof parallels which of \cite[Proposition 4.3]{CT16}. Fix any constant $\lambda\in (0,\Lambda(\mathcal{G}^1,5\eta,t))$, and let $E_t^{5\eta}$ be a $(t,5\eta)$-separated set for $(\mathcal{G}^1)_t$ satisfying
    $$
\sum_{x\in E_t^{5\eta}}e^{\Phi(x,t)}>\lambda.
    $$
    Let $(E_t^{5\eta})^n$ be the self-product of $E_t^{5\eta}$ for $n$ times. By applying weak controlled specification over $\mathcal{G}^1$ at scale $\eta$, for any $\mathsf{x}:=\{x_i\}_{i=1}^n\in (E_t^{5\eta})^n$, there exist $\{\tau_i\}_{i=1}^{n-1}\in [0,h(t)]^{n-1}$ and $y=y(\mathsf{x})=\text{Spec}_{\{\tau_i\}}^{n,\eta}(\{(x_i,t)\})$. Here we may consider $y$ as a map mapping $(E_t^{5\eta})^n$ to $X$. Write $T_n:=tn+(n-1)h(t)$. Let $E_n$ be a maximally separated $(\eta,T_n)$-separated set for $X$, and $p:X\to E_n$ be the map that sends $x\in X$ to its closest point in $E_n$ within $d_{T_n}$-metric. Let $\pi:=p\circ y$.

    Let $m=m(\eta)\geq 2$ be an integer such that 
    $$\max_{x\in X,s\in [0,\frac{1}{m}]}\{d(x,f_s(x))\}<\eta.$$ 
    Partition $[0,(n-1)h(t)]$ into $(n-1)mh(t)$ sub-intervals $\{I_j\}_{j=1}^{(n-1)mh(t)}$ with equal length. For each $\mathsf{x}\in (E_t^{5\eta})^n$, let $l(\mathsf{x})=\{j_i\}_{i=1}^{n-1}$ be such that $\sum_{i=1}^{k}\tau_i\in I_{j_k}$ for all $k\in \{1,\cdots,n-1\}$. Notice that the set $\{l(\mathsf{x}):\mathsf{x}\in (E_t^{5\eta})^n\}$ consists of at most $(mh(t)+1)^{n-1}$ elements. Following the proof of \cite[Lemma 4.4]{CT16}, if $l(\mathsf{x}_1)=l(\mathsf{x}_2)$, then $\pi(\mathsf{x}_1)\neq \pi(\mathsf{x}_2)$. Therefore, for each $z\in E_n$, we have
    \begin{equation} \label{eqpreimagecardinality}
        \# \pi^{-1}(z)\leq (mh(t)+1)^{n-1}.
    \end{equation}
    In the mean time, following the proof of \cite[Lemma 4.5]{CT16}, for each $\mathsf{x}\in (E_t^{5\eta})^n$, we have
    \begin{equation} \label{eqvariation1}
        \prod_{i=1}^n e^{\Phi(x_i,t)}\leq e^{ng(t)+2(n-1)h(t)||\ph||}e^{\Phi(\pi(\mathsf{x}),T_n)}.
    \end{equation}
    Adding the LHS of \eqref{eqvariation1} over all possible $\mathsf{x}$, we have
    $$
\begin{aligned}
    \lambda^n
    &<(\Lambda(\mathcal{G}^1,5\eta,t))^n \\
    &=\sum_{\mathsf{x}}\prod_{i=1}^n e^{\Phi(x_i,t)} \\
    &\leq e^{ng(t)+2(n-1)h(t)||\ph||}\sum_{\mathsf{x}}e^{\Phi(\pi(\mathsf{x}),T_n)} \\
    &\leq e^{ng(t)+2(n-1)h(t)||\ph||}(mh(t)+1)^{n-1}\Lambda(X,\eta,T_n) \\
    &< e^{ng(t)+(n-1)(2||\ph||+m)h(t)}\Lambda(X,\eta,T_n) \\
\end{aligned}
    $$
    where the inequality on the fourth line follows from \eqref{eqpreimagecardinality}. Writing $C_1:=2||\ph||+m$, we conclude the proof of the lemma.
\end{proof}
By taking logarithm, dividing by $nt$, making $n\to \infty$ on both sides of \eqref{equpperboundmultipleG} and applying Proposition \ref{propexpansivegood}(1), we have
\begin{corollary}
    For every $t>0$, we have
\begin{equation} \label{equpperboundGonescale}
    \Lambda(\mathcal{G}^1,5\eta,t)< e^{C_1h(t)+g(t)}e^{(t+h(t))P}.
\end{equation}
\end{corollary}
It follows immediately from \eqref{equpperboundGonescale}, $C_1>2$, and the fact of $\ph$ being $g$-distorted at scale $\eps$ that
\begin{equation} \label{equpperboundGtwoscale}
    \Lambda(\mathcal{G}^1,5\eta,\eps,t)<
e^{C_1(h(t)+g(t))}e^{(t+h(t))P}
\end{equation}

In view of the assumption of Theorem \ref{thm:a}, let $\{t_k\}_{k\in \mathbb{N}}$ be an increasing sequence of positive numbers satisfying
\begin{equation} \label{seq}
\lim_{k\to \infty}\frac{h(t_k)+g(t_k)}{\log t_k}=0.
\end{equation}
For any $k\in \mathbb{N}$, we say a time $t>0$ is $k$-iconic if $t\in [\frac{t_k}{4},t_k]$. We also say $t$ is iconic if it is $k$-iconic for some $k$. For a collection of orbit segments with length being $k$-iconic, when $k$ is large enough, we will show that amounts of orbital information is occupied by those segments whose major part is in $\mathcal{G}$, as stated in the following lemma.
\begin{lemma} \label{lemmaGmajor}
    Let $\alpha:\mathbb{R}^+\to \mathbb{R}$ be a monotonic function satisfying $$\liminf_{t\to \infty}\frac{\alpha(t)}{h(t)+g(t)}>-\infty.$$ Then there exist constants $n_{\alpha}\in \mathbb{N}$ and $C_2=C_2(\alpha)$ such that the following holds true: for every $n>n_{\alpha}$, $t>0$ being $n$-iconic, and $\mathcal{C}\subset X\times \mathbb{R}^+$ such that $\Lambda(\mathcal{C},10\eta,\frac{\eps}{2},t)>e^{tP+\alpha(t)}$, we have
    $$
\Lambda(\mathcal{C}\cap\mathcal{G}^{C_2(h+g)},10\eta,\frac{\eps}{2},t)>\frac{e^{tP+\alpha(t)}}{2}.
    $$
\end{lemma}
\begin{proof}
    Fix $\delta\in (0,P-P(\mathcal{D}^c\cup [\mathcal{P}] \cup [\mathcal{S}],\ph,5\eta,\eps))$. Let $C_{5\eta}>0$ be such that
\begin{equation} \label{eqbadpartitionsum}
    \Lambda_t(\mathcal{D}^c\cup [\mathcal{P}] \cup [\mathcal{S}],\ph,5\eta,\eps)<C_{5\eta}e^{t(P-\delta)} \text{ for all }t>0
\end{equation}
    and $C_{\alpha}\in \mathbb{R}$ be such that
    $$
\alpha(t)>C_{\alpha}(h(t)+g(t)) \text{ for all }t>0.
    $$
    Write
    $$
C_2=C_2(\alpha):=\frac{C_1+|P|+|C_{\alpha}|+16(C_{5\eta})^2/(1-e^{-\delta})^2}{\delta}
    $$
    where $C_1$ is from Lemma \ref{lemmaGupperbound}. Choose $n_{\alpha}$ large enough such that for all $n>n_{\alpha}$, we have
    \begin{equation} \label{eq:nalphacondition}
        \log 4(C_{5\eta})^2< \delta t_n/4-|C_{\alpha}|(h(t_n)+g(t_n)).
    \end{equation}
  Let $n>n_{\alpha}$, $t$ be $n$-iconic, and $E_t$ be a $(10\eta,t)$-separated set for $\mathcal{C}_t$ satisfying
\begin{equation} \label{eqEtchoice}
    \sum_{x\in E_t}e^{\Phi_{\frac{\eps}{2}}(x,t)}>e^{tP+\alpha(t)}. 
\end{equation}
    If $x\in E_t\cap \mathcal{D}_t$, write $p(x):=\lfloor p(x,t) \rfloor$, $s(x):=\lfloor s(x,t) \rfloor$, where $\lfloor c \rfloor$ denotes the greatest integer not greater than $c$. Also write $g(x):=t-p-s$. Observe that we have $(x,p(x))\in [\mathcal{P}]$, $(f_{p(x)}(x),g(x))\in \mathcal{G}^1$, and $(f_{p(x)+g(x)}(x),s(x))\in [\mathcal{S}]$. For any pair of non-negative integers $(p,s)$ satisfying $p+s\leq \lfloor t \rfloor$, write
    $$
E(p,s):=\{x\in E_t\cap \mathcal{D}_t: p(x)=p,s(x)=s\}.
    $$
    For each integer $i\in [0,\lfloor t \rfloor]$, let $P_i$ (\emph{resp.} $S_i$) be a maximal $(5\eta,i)$-separated set for $[\mathcal{P}]_i$ (\emph{resp.} $[\mathcal{S}]_i$). We also define $G_u$ similarly w.r.t $(\mathcal{G}^1)_u$ for general $u\in [0,\lfloor t \rfloor]$. Then we can define a map $\pi:E(p,s) \to P_p\times G_{t-p-s} \times S_s$ in the form of $\pi(x)=(x_p,x_g,x_s)$ satisfying
    $$
\max\{d_p(x,x_p),d_{t-p-s}(f_p(x),x_g),d_s(f_{t-s}(x),x_s)\}<5\eta.
    $$
    Observe that the map $\pi$ is injective. Meanwhile, for every such $x\in E(p,s)$, we have
    \begin{equation} \label{eqvariationalongxt}
\begin{aligned}
        \Phi_{\frac{\eps}{2}}(x,t)
        &\leq \Phi_{\frac{\eps}{2}+5\eta}(x_p,p)+\Phi_{\frac{\eps}{2}+5\eta}(x_g,t-p-s)+\Phi_{\frac{\eps}{2}+5\eta}(x_s,s)\\
        &\leq \Phi_{\eps}(x_g,t-p-s)+\Phi_{\eps}(x_p,p)+\Phi_{\eps}(x_s,s).
\end{aligned}
    \end{equation}
    Then, by writing $p\vee s:=\min\{p,s\}$ and $a(t):=C_2(h(t)+g(t))$, we have
\begin{equation} \label{eqEtsum}
\begin{aligned}
    &\sum_{x\in E_t}e^{\Phi_{\frac{\eps}{2}}(x,t)} \\
   = &(\sum_{x\in E_t\cap (\mathcal{D}^c)_t}+\sum_{p,s\in \mathbb{N}:p+s\leq \lfloor t \rfloor}\sum_{x\in E_{p,s}})e^{\Phi_{\frac{\eps}{2}}(x,t)} \\
    \leq  & \Lambda_t(\mathcal{D}^c,\ph,10\eta,\eps)+\sum_{p\vee s\leq a(t)}e^{\Phi_{\frac{\eps}{2}}(x,t)}+\sum_{p\vee s> a(t)}e^{\Phi_{\eps}(x_g,t-p-s)+\Phi_{\eps}(x_p,p)+\Phi_{\eps}(x_s,s)} \\
     \leq  &\Lambda_t(\mathcal{D}^c,\ph,10\eta,\eps)+\sum_{p\vee s\leq a(t)}e^{\Phi_{\frac{\eps}{2}}(x,t)}\\
     &+\sum_{p\vee s> a(t)}\Lambda_p([\mathcal{P}],\ph,5\eta,\eps)\Lambda_{t-p-s}(\mathcal{G}^1,\ph,5\eta,\eps)\Lambda_s([\mathcal{S}],\ph,5\eta,\eps) \\
     < &C_{5\eta}e^{t(P-\delta)}+\sum_{p\vee s\leq a(t)}e^{\Phi_{\frac{\eps}{2}}(x,t)}+\sum_{p\vee s> a(t)}C_{5\eta}^2e^{(p+s)(P-\delta)}e^{C_1(h(t)+g(t))}e^{(t-p-s+h(t))P} \\
     < &\sum_{p\vee s\leq a(t)}e^{\Phi_{\frac{\eps}{2}}(x,t)}+(C_{5\eta})^2e^{tP}(e^{-\delta t}+\sum_{p\vee s> a(t)}e^{h(t)(C_1+P)+C_1g(t)}e^{-(p+s)\delta}).
\end{aligned}
\end{equation}
Notice that the last summand on the last line, which is 
$$\sum_{p\vee s> a(t)}e^{h(t)(C_1+P)+C_1g(t)}e^{-(p+s)\delta},$$
can be controlled from above by
$$
\begin{aligned}
 &2e^{h(t)(C_1+P)+C_1g(t)}e^{-\lceil a(t) \rceil \delta}\sum_{j=0}^{\infty}(j+1)e^{-j\delta}\\
 <&\frac{2e^{-\lceil a(t) \rceil \delta+h(t)(C_1+P)+C_1g(t)}}{(1-e^{-\delta})^2}.   
\end{aligned}
$$
We claim that
\begin{equation} \label{eqexponential1}
    (C_{5\eta})^2e^{-\delta t}<\frac{e^{-|C_{\alpha}|
    (h(t)+g(t))}}{4}\leq \frac{e^{C_{\alpha}(h(t)+g(t))}}{4}<\frac{e^{\alpha(t)}}{4}
\end{equation}
and
\begin{equation} \label{eqexponential2}
    \frac{2(C_{5\eta})^2e^{(C_1+|P|)(h(t)+g(t))-\delta \lceil a(t) \rceil}}{(1-e^{-\delta})^2}<\frac{e^{C_{\alpha}(h(t)+g(t))}}{4}<\frac{e^{\alpha(t)}}{4}.
\end{equation}
We show how the lemma is concluded by assuming \eqref{eqexponential1} and \eqref{eqexponential2}. Applying \eqref{eqEtchoice}, we may continue \eqref{eqEtsum} as
$$
e^{tP+\alpha(t)}<\sum_{x\in E_t}e^{\Phi_{\frac{\eps}{2}}(x,t)}<\sum_{p\vee s\leq a(t)}e^{\Phi_{\frac{\eps}{2}}(x,t)}+(\frac{e^{\alpha(t)}}{4}+\frac{e^{\alpha(t)}}{4})e^{tP},
$$
which implies that $\Lambda_t(\mathcal{C}\cap \mathcal{G}^{a(t)},10\eta,\frac{\eps}{2})\geq\sum_{p\vee s\leq a(t)}e^{\Phi_{\frac{\eps}{2}}(x,t)}>\frac{e^{tP+\alpha(t)}}{2}$.

It remains to prove \eqref{eqexponential1} and \eqref{eqexponential2}. For \eqref{eqexponential1}, by our choice of $t$ being $n$-iconic, it suffices to show that
$$
e^{-\frac{\delta t_n}{4}}<\frac{e^{-|C_{\alpha}|(h(t_n)+g(t_n))}}{4(C_{5\eta})^2},
$$
which is true due to \eqref{eq:nalphacondition}. Meanwhile, \eqref{eqexponential2} follows immediately from definition of $a(t)$ and our choice on $C_2$. The lemma is then concluded.
\end{proof}

As defined in \cite[(6.1)]{PYY22}, given any $\mathcal{C}\subset X\times \mathbb{R}^+$ and $i,j\in \mathbb{N}$, we denote by
$$
f_{i,j}(\mathcal{C}):=\{(f_i(x),t-i-j):(x,t)\in \mathcal{C},t\geq i+j\}
$$
the collection of orbit segments formed by removing the initial and ending parts with integer length from which of $\mathcal{C}$. The above lemma then has the following consequence.

\begin{corollary} \label{coro:LambdaG1lowerbound}
    Let $\alpha$, $C_2$ and $n_{\alpha}$ be as in the preceding lemma. There exists a constant $C_3=C_3(\eta)$ such that the following holds true: for every $n>n_{\alpha}$, $n$-iconic $t$, and $\mathcal{C}\subset X\times \mathbb{R}^+$ satisfying  $\Lambda_t(\mathcal{C},10\eta,\frac{\eps}{2})>e^{tP+\alpha(t)}$, there exist integers $i_{\alpha},j_{\alpha}\in [0,C_2(h(t)+g(t))]$ such that
    $$
\Lambda_{t-i_{\alpha}-j_{\alpha}}(f_{i_{\alpha},j_{\alpha}}(\mathcal{C})\cap \mathcal{G}^1,10\eta,\frac{\eps}{2})>\frac{C_3e^{(t-i_{\alpha}-j_{\alpha})P+\alpha(t)}}{(h(t)+g(t))^2}.
    $$
\end{corollary}
\begin{proof}
    Let $n,t,\mathcal{C}$ be as in the statement. Following the preceding lemma, there exists a $(10\eta,t)$-separated set $E_t^<$ for $\mathcal{C}$ satisfying
    \begin{enumerate}
        \item $E_t^<\subset \mathcal{D}_t$ and $p(x,t)\vee s(x,t)\leq C_2(h(t)+g(t))$ for all $x\in E_t^<$,
        \item $\sum_{x\in E_t^<}e^{\Phi_{\frac{\eps}{2}}(x,t)}>\frac{e^{tP+\alpha(t)}}{2}$.
    \end{enumerate}
    Since there are at most $(1+\lfloor C_2(h(t)+g(t)) \rfloor)^2$ choices on the pair $(p(x,t),s(x,t))$, there must be a pair of integers $(p,s)\in [0,\lfloor C_2(h(t)+g(t)) \rfloor]^2$ such that
\begin{equation} \label{eqEt<partitionsum}
    \sum_{x\in E_t^<(p,s)}e^{\Phi_{\frac{\eps}{2}}(x,t)}>\frac{e^{tP+\alpha(t)}}{2(1+C_2(h(t)+g(t)))^2}>\frac{e^{tP+\alpha(t)}}{8(C_2(h(t)+g(t)))^2}.
\end{equation}
    where $E_t^<(p,s)$ is defined similarly as $E(p,s)$ using $E_t^<$ in the place of $E$. Let $P_p$ (resp. $S_s$) be a maximal $(5\eta,p)$-separated set for $[\mathcal{P}]_p$ (resp. $(5\eta,s)$-separated set for $[\mathcal{S}]_s$). Let $\pi_{\mathcal{P},\mathcal{S}}:E_t^<(p,s)\to P_p\times S_s$ be the collection of the first and the third components of map $\pi$ in the proof of Lemma \ref{lemmaGmajor}, i.e. for any $x\in E_t^<(p,s)$, $\pi_{\mathcal{P},\mathcal{S}}(x)=(x_p,x_s)$ satisfying
    $$
\max\{d_p(x,x_p),d_s(f_{t-s}(x),x_s)\}<5\eta.
    $$
    Observe that for $x_1,x_2$ sharing the same image under $\pi_{\mathcal{P},\mathcal{S}}$, 
    $$d_{t-p-s}(f_p(x_1),f_p(x_2))>10\eta,$$ 
    and $f_p(x)\in (\mathcal{G}^1)_{t-p-s}$. Then we may continue \eqref{eqEt<partitionsum} as
    \begin{equation}
        \begin{aligned}
        &\frac{e^{tP+\alpha(t)}}{8(C_2(h(t)+g(t)))^2} \\
            <&\sum_{x\in E_t^<(p,s)}e^{\Phi_{\frac{\eps}{2}}(x,t)} \\
            =&\sum_{u_p\in P_p,u_s\in S_s}\sum_{x\in \pi_{\mathcal{P},\mathcal{S}}^{-1}(u_p,u_s)}e^{\Phi_{\frac{\eps}{2}}(x,t)}\\
            \leq &\sum_{u_p\in P_p,u_s\in S_s} e^{\Phi_{\frac{\eps}{2}+5\eta}(u_p,p)+\Phi_{\frac{\eps}{2}+5\eta}(u_s,s)}\sum_{x\in \pi_{\mathcal{P},\mathcal{S}}^{-1}(u_p,u_s)}e^{\Phi_{\frac{\eps}{2}}(f_p(x),t-p-s)} \\
            \leq& \sum_{u_p\in P_p,u_s\in S_s}e^{\Phi_{\eps}(u_p,p)+\Phi_{\eps}(u_s,s)}\Lambda_{t-p-s}(f_{p,s}(\mathcal{C})\cap \mathcal{G}^1,10\eta,\frac{\eps}{2}) \\
            \leq& \Lambda_p(\mathcal{P},5\eta,\eps)\Lambda_s(\mathcal{S},5\eta,\eps)\Lambda_{t-p-s}(f_{p,s}(\mathcal{C})\cap \mathcal{G}^1,10\eta,\frac{\eps}{2}) \\
            <& (C_{5\eta})^2e^{(p+s)(P-\delta)}\Lambda_{t-p-s}(f_{p,s}(\mathcal{C})\cap \mathcal{G}^1,10\eta,\frac{\eps}{2}) \\
            \leq &(C_{5\eta})^2e^{(p+s)P}\Lambda_{t-p-s}(f_{p,s}(\mathcal{C})\cap \mathcal{G}^1,10\eta,\frac{\eps}{2}).
        \end{aligned}
    \end{equation}
    Letting $i_{\alpha}=:p$, $j_{\alpha}:=s$, and $C_3:=\frac{1}{8(C_2C_{5\eta})^2}$, the corollary is concluded.
\end{proof}
A combination of \eqref{equpperboundGtwoscale} and Corollary \eqref{coro:LambdaG1lowerbound} provides us with the following upper bound concerning the partition sum over $X$.
\begin{lemma} \label{lemmaXpartitionsum}
    There exists a constant $C_0=C_0(\eta)>0$ such that for all $t>0$ being iconic, we have
    $$
\Lambda_t(X,10\eta,\frac{\eps}{2})\leq e^{C_0(h(t)+g(t))+tP}.
    $$
\end{lemma}

\subsection{Lower bound of partition sums}
The following result provides us with a key uniform lower bound of partition sums over sets with positive measure w.r.t any given equilibrium state for $(X,F,\ph)$. It is in this place where we are forced to use the two-scale partition sum involving both $\eta$ and $\eps$.
\begin{proposition} \label{propkey}
    For every $\beta>0$, there is a constant $C_{\beta}>0$ such that the following is true. Let $\mu$ be an equilibrium state for $(X,F,\ph)$, and $A\subset X$ be a set satisfying $\mu(A)>\beta$. Then for every iconic $t>0$ and  maximal $(10\eta,t)$-separated set $E^A_{10\eta,t}$ for $A$, we have
    $$
\sum_{x\in E^A_{10\eta,t}}e^{\Phi_{10\eta}(x,t)}\geq e^{-C_{\beta}(h(t)+g(t))+tP}.
    $$
    In particular, we have
    $$
\Lambda_t(A,10\eta,\frac{\eps}{2})\geq e^{-C_{\beta}(h(t)+g(t))+tP}.
    $$
\end{proposition}
\begin{proof}
    The main idea of the proof parallels which of \cite[Lemma 4.18]{CT16}, yet some necessary adaptions have to be made on the construction of adapted partition. Let $E^A_{10\eta,t}$ be as in the statement. We write out its elements as $\{u_1,\cdots,u_m\}$, where $m=\#E^A_{10\eta,t}$. Expand $E^A_{10\eta,t}$ to a maximal $(10\eta,t)$-separated set for $X$ by adding finitely many points $\{u_{m+1},\cdots,u_l\}$. Notice that $X\subset \bigcup_{i=1}^l B_t(u_i,10\eta)$. The partition $\xi_t:=\{\omega_i\}_{i=1}^l$ is constructed in a way such that $\bigcup_{i=1}^k \omega_i\subset \bigcup_{i=1}^k B_t(u_i,10\eta)$ for all $k\in [1,l]\cap \mathbb{N}$ and $u_i\in \omega_i$, which is always possible by taking $\omega_1:=B_t(u_1,10\eta)$ and $\omega_{k+1}:=B_t(u_{k+1},10\eta)\setminus (\cup_{i=1}^k \omega_i)$.

    Since the diameter of elements in $\omega_t$ in $d_t$-metric is no greater than $20\eta$, which is by assumption less than $\frac{\eps}{2}$, we know from Proposition \ref{propexpansivegood} (2) and Abramov's formula that $h_{\mu}(f_t,\xi_t)=h_{\mu}(f_t)=th_{\mu}(f_1)$. Then by abbreviating $\cup_{i=i_1}^{i_2} \omega_i$ as $\omega_{i_1}^{i_2}$ and
    following the proof in \cite[Lemma 3.8, 4.18]{CT16}, we have
    $$
\begin{aligned}
    &tP=tP_{\mu}(\ph)\\
    =&h_{\mu}(f_t)+\int \Phi_0(x,t)d\mu \\
    \leq& \sum_i \mu(\omega_i)(-\log \mu(\omega_i)+\Phi_{10\eta}(u_i,t)) \\
    \leq & \mu(\omega_1^m)\log \sum_{i=1}^m e^{\Phi_{10\eta}(u_i,t)}
    +\mu(\omega_{m+1}^{l})\log \sum_{i=m+1}^l e^{\Phi_{10\eta}(u_i,t)}+\log 2 \\
    < &\mu(\omega_1^m)\log \sum_{i=1}^m e^{\Phi_{10\eta}(u_i,t)}+(1-\mu(\omega_1^m))\log \Lambda_t(X,10\eta,\frac{\eps}{2})+\log 2 \\
    \leq &\mu(\omega_1^m)\log \sum_{i=1}^m e^{\Phi_{10\eta}(u_i,t)}+(1-\mu(\omega_1^m))(C_0(h(t)+g(t))+tP)+\log 2,
\end{aligned}
    $$
    where the last inequality follows from Lemma \ref{lemmaXpartitionsum}. After rearrangement, we may continue from above as
    $$
\begin{aligned}
    \log \sum_{i=1}^m e^{\Phi_{10\eta}(u_i,t)}
    &>tP-\frac{(1-\mu(\omega_1^m))C_0(h(t)+g(t))}{\mu(\omega_1^m)}-\frac{\log 2}{\mu(\omega_1^m)} \\
    &>tP-\frac{(1-\beta)C_0(h(t)+g(t))}{\beta}-\frac{\log 2}{\beta}.
\end{aligned}
    $$
By choosing any $C_{\beta}>\max\{\frac{2(1-\beta)C_0}{\beta},\frac{\log 2}{\beta}\}$, we have 
$$\log \sum_{i=1}^m e^{\Phi_{10\eta}(u_i,t)}>tP-C_{\beta}(h(t)+g(t)),$$ 
which concludes the proof of the proposition.
\end{proof}

\begin{proof}[Proof of Theorem \ref{thm:weprove}]
    We are now ready to proceed to the proof of Theorem \ref{thm:weprove}. Assume by contradiction that there exist two ergodic equilibrium states $\mu$ and $\nu$ for $(X,F,\ph)$. Let $E_{\mu}$ be the set of $\mu$-generic points. Then $\mu(E_{\mu})=1$, $\nu(E_{\mu})=0$. Let $G_{\mu}\subset E_{\mu}$ be a compact set satisfying $\mu(G_{\mu})\in (\frac{9}{10},1)$, and $G_{\nu}\subset (E_{\mu})^c$ be a compact set satisfying $\nu(G_{\nu})\in (\frac{9}{10},1)$. Writing $B_{\mu}:=(G_{\mu})^c$ and $B_{\nu}:=(G_{\nu})^c$, we know $\max\{\mu(B_{\mu}),\nu(B_{\nu})\}\in (0,\frac{1}{10})$. Let 
    $$
\eta:=\min\{\frac{d(G_{\mu},G_{\nu})}{7},\frac{\eps}{100}\}.
    $$
    By applying \eqref{eqM-g} with $g:=\chi_{B_{\mu}}$, $\lambda=\frac{1}{5}$, and \eqref{eqM+g} with $g:=\chi_{B_{\nu}}$, $\lambda=\frac{1}{5}$, and writing 
    $$G^+_{\nu}:=\{x\in X:(M^+\chi_{B_{\nu}})(x)\leq \frac{1}{5}\},$$ 
    and
    $$G^-_{\mu}:=\{x\in X:(M^-\chi_{B_{\mu}})(x)\leq \frac{1}{5}\},$$ 
    we have
    $$
\nu(G^+_{\nu})>\frac{1}{2}, \quad \quad \quad \quad \mu(G^-_{\mu})>\frac{1}{2}.
    $$
    
    Let $k\in \mathbb{N}$ be sufficiently large, and $t_{\mu}^0=t_{\nu}^0=\frac{t_k}{3}$ where $t_k$ is from \eqref{seq}. Let $t_G:=4C_2(h(t_k)+g(t_k))$, and $T_G:=\lfloor \frac{t_k}{12t_G} \rfloor$. For each integer $l\in [0,T_G]$, let
$$
t_{\mu}^l:=t_{\mu}^0-lt_G, \quad t_{\nu}^l:=t_{\nu}^0+lt_G.
$$
Notice that $t_{\mu}^l\geq \frac{t_k}{4}$ for all $l$. By Proposition \ref{propkey}, there exists a constant $C_{\frac{1}{2}}>0$ such that
$$
\Lambda_{t_{\mu}^l}(f_{-t_{\mu}^l}(G_{\mu}^-),10\eta,\frac{\eps}{2})
>e^{-C_{\frac{1}{2}}(h(t_k)+g(t_k))+t_{\mu}^lP},
$$
and
$$
\Lambda_{t_{\nu}^l}(G_{\nu}^+,10\eta,\frac{\eps}{2})
>e^{-C_{\frac{1}{2}}(h(t_k)+g(t_k))+t_{\nu}^lP}.
$$
Then for each $l$, by Corollary \ref{coro:LambdaG1lowerbound}, there exist $i_{\mu}^l,j_{\mu}^l,i_{\nu}^l,j_{\nu}^l \in [0,C_2(h(t_k)+g(t_k))]$, a set $E_{\mu}^l$ being $(10\eta,t_{\mu}^l-i_{\mu}^l-j_{\mu}^l)$-separated for $(\mathcal{G}^1\cap f_{i_{\mu}^l,j_{\mu}^l}(f_{-t_{\mu}^l}(G_{\mu}^-),t_{\mu}^l))_{t_{\mu}^l-i_{\mu}^l-j_{\mu}^l}$, and respectively a set $E_{\nu}^l$ being $(10\eta,t_{\nu}^l-i_{\nu}^l-j_{\nu}^l)$-separated for $(\mathcal{G}^1\cap f_{i_{\nu}^l,j_{\nu}^l}(G_{\nu}^+,t_{\nu}^l))_{t_{\nu}^l-i_{\nu}^l-j_{\nu}^l}$ such that
    $$
\sum_{x_{\mu}^l\in E_{\mu}^l}e^{\Phi_{\frac{\eps}{2}}(x_{\mu}^l,t_{\mu}^l-i_{\mu}^l-j_{\mu}^l)} >\frac{C_3e^{(t_{\mu}^l-i_{\mu}^l-j_{\mu}^l)P-C_{\frac{1}{2}}(h(t_k)+g(t_k))}}{(h(t_k)+g(t_k))^2}
    $$
    and
    $$
\sum_{y_{\nu}^l\in E_{\nu}^l}e^{\Phi_{\frac{\eps}{2}}(y_{\nu}^l,t_{\nu}^l-i_{\nu}^l-j_{\nu}^l)}>\frac{C_3e^{(t_{\nu}^l-i_{\nu}^l-j_{\nu}^l)P-C_{\frac{1}{2}}(h(t_k)+g(t_k))}}{(h(t_k)+g(t_k))^2}.
    $$
    Write $r_{\mu}^l:=t_{\mu}^l-i_{\mu}^l-j_{\mu}^l$ and $r_{\nu}^l:=t_{\nu}^l-i_{\nu}^l-j_{\nu}^l$. For each $(x_{\mu}^l,y_{\nu}^l)\in E_{\mu}^l \times E_{\nu}^l$, there exist some $\tau^l=\tau^l((x_{\mu}^l,y_{\nu}^l))\in [0,h(t_k)]$ and $$z^l=z^l((x_{\mu}^l,y_{\nu}^l)):=\text{Spec}^{2,\eta}_{\tau^l}((x_{\mu}^l,r_{\mu}^l),(y_{\nu}^l,r_{\nu}^l))\in X.$$ 
    Writing $t(z^l):=r_{\mu}^l+\tau^l+r_{\nu}^l$, we have
    \begin{equation}
\begin{aligned}
&\sum_{x_{\mu}^l,y_{\nu}^l}e^{\Phi(z^l,t(z^l))}\\
>&\sum_{x_{\mu}^l,y_{\nu}^l}e^{\Phi_{\frac{\eps}{2}}(x_{\mu}^l,r_{\mu}^l)-h(t_k)||\ph||+\Phi_{\frac{\eps}{2}}(x_{\nu}^l,r_{\nu}^l)} \\
>&\frac{C_3^2e^{(r_{\mu}^l+r_{\nu}^l)P}e^{-2C_{\frac{1}{2}}(h(t_k)+g(t_k))}}{(h(t_k)+g(t_k))^4e^{h(t_k)||\ph||}} \\
\geq &\frac{C_3^2e^{(t_{\mu}^l+t_{\nu}^l)P}e^{-(2C_{\frac{1}{2}}+4C_2)(h(t_k)+g(t_k))}}{(h(t_k)+g(t_k))^4e^{h(t_k)||\ph||}} \\
=&\frac{C_3^2e^{\frac{2t_kP}{3}}e^{-(2C_{\frac{1}{2}}+4C_2)(h(t_k)+g(t_k))}}{(h(t_k)+g(t_k))^4e^{h(t_k)||\ph||}}.
\end{aligned}
    \end{equation}
    Adding over all possible $l\in [0,T_G]$ and writing $C':=48C_2/C_3^2$, we have
    $$
\sum_{l}\sum_{z^l}e^{\Phi(z,t(z^l))}>\frac{t_ke^{\frac{2t_kP}{3}}e^{-(2C_{\frac{1}{2}}+4C_2+||\ph||)(h(t_k)+g(t_k))}}{C'(h(t_k)+g(t_k))^5}.
    $$
    Notice that $$t(z^l)\in [\frac{2t_k}{3}-4C_2(h(t_k)+g(t_k)),\frac{2t_k}{3}+h(t_k)].$$ 
    Also recall $m=m(\eta)$ from the proof of Lemma \ref{lemmaGupperbound}. Divide 
    $$[\frac{2t_k}{3}-4C_2(h(t_k)+g(t_k)),\frac{2t_k}{3}+C_2(h(t_k)+g(t_k))]$$ 
    into $5mC_2(h(t_k)+g(t_k))$ sub-intervals with equal length $\frac{1}{m}$. Denote the collection of such intervals by $\{I_q\}_{q=1}^{5mC_2(h(t_k)+g(t_k))}$, and the collection of $z^l$ satisfying $t(z^l)\in I_q$ by $Z_q$. Therefore, there exists some integer 
    $$q_1\in [1,5mC_2(h(t_k)+g(t_k))]$$ 
    such that
\begin{equation} \label{eqZq1large}
    \sum_{z^l\in Z_{q_1}}e^{\Phi(z,t(z^l))}>\frac{t_ke^{\frac{2t_kP}{3}}e^{-(2C_{\frac{1}{2}}+4C_2+||\ph||)(h(t_k)+g(t_k))}}{5mC_2C'(h(t_k)+g(t_k))^6}.
\end{equation}
    We make the following key claim, which plays an essential role in providing us with the evidence to contradiction.

    \textbf{Claim:} For each $q\in [1,5mC_2(h(t_k)+g(t_k))]$, let $E_q$ be the collection of $$\{(x_{\mu}^l,y_{\nu}^l):x_{\mu}^l\in E_{\mu}^{l},y_{\nu}^l\in E_{\nu}^{l}\}$$ such that $$t(z^l)=t(z^l(x_{\mu}^l,y_{\nu}^l))\in I_q.$$ Then the map $z:E_q\to Z_q$ sending $(x_{\mu}^l,y_{\nu}^l)$ to $z^l(x_{\mu}^l,y_{\nu}^l)$ is injective. In fact, writing 
    $$s_{q}:=\frac{2t_k}{3}-4C_2(h(t_k)+g(t_k))+\frac{q}{m},$$ 
    $Z_q$ is $(5\eta,s_q)$-separated.

    We first explain how the above claim leads to a contradiction. Together with \eqref{eqZq1large}, the claim above implies that
    $$
\Lambda_{s_{q_1}}(X,5\eta,0)>\frac{s_{q_1}e^{-(2C_{\frac{1}{2}}+8C_2+||\ph||)(h(t_k)+g(t_k))}e^{s_{q_1}P}}{5mC_2C'(h(t_k)+g(t_k))^6}.
    $$
    By definition of $\{t_k\}$, when $k$ is large enough, we have $\frac{h(t_k)+g(t_k)}{\log t_k}$ is sufficiently close to $0$, in a sense that
    \begin{enumerate}
        \item $(2C_{\frac{1}{2}}+8C_2+||\ph||)(h(t_k)+g(t_k))<\frac{\log t_k}{100}$,
        \item $5mC_2C'(h(t_k)+g(t_k))^6<t_k^{\frac{1}{100}}$,
        \item $C_0(h(t_k)+g(t_k))<\frac{t_k}{100}$.
    \end{enumerate}
Then 
\begin{equation*}
\begin{aligned}
&\Lambda_{s_{q_1}}(X,5\eta,0)>t_k^{-\frac{1}{50}}s_{q_1}e^{s_{q_1}P}>(2s_{q_1})^{-\frac{1}{50}}s_{q_1}e^{s_{q_1}P}\\
>&e^{C_0(h(t_k)+g(t_k))+s_{q_1}P}\geq e^{C_0(h(s_{q_1})+g(s_{q_1}))+s_{q_1}P}, 
\end{aligned}   
\end{equation*}
contradicting Lemma \ref{lemmaXpartitionsum}.

It remains to show that the claim holds true. Given $(x_i,y_i)\in E_q\cap (E_{\mu}^{l_i}\times E_{\nu}^{l_i})$ with $i=1,2$ satisfying $z_1:=z(x_1,y_1)=z_2:=z(x_2,y_2)$, we show that $l_1=l_2$, $(x_1,y_1)=(x_2,y_2)$. We first look at the case where $l_1=l_2=l$. Notice that $x_1$ must be equal to $x_2$, as otherwise we know that $d_{r_{\mu}^l}(x_1,x_2)\geq 10\eta$, which implies that $d_{r_{\mu}^l}(z_1,z_2)\geq 8\eta$, contradicting $z_1=z_2$. In the mean time, if $y_1\neq y_2$, since both $t(z_1),t(z_2)$ are in $I_q$, writing $\tau_i:=\tau_{l}(x_i,y_i)$ with $i\in \{1,2\}$, $|\tau_1-\tau_2|\leq\frac{1}{m}$. Without loss of generality, we assume $\tau_1\geq \tau_2$. Then 
\begin{equation}
\begin{aligned}
&d_{r_{\nu}^l-\tau_1+\tau_2}(f_{r_{\mu}^l+\tau_1}(z_1),f_{r_{\mu}^l+\tau_1}(z_2))\geq d_{r_{\nu}^l-\tau_1+\tau_2}(y_1,f_{\tau_1-\tau_2}(y_2))-2\eta\\
\geq &d_{r_{\nu}^l-\tau_1+\tau_2}(y_1,y_2)-3\eta\geq 7\eta, 
\end{aligned}    
\end{equation}
reaching out to a contradiction.

Now we turn to the case where $l_1\neq l_2$. Assume without loss of generality that $l_1>l_2$. Since $y_1\in (f_{i_{\nu}^{l_1},j_{\nu}^{l_1}}(G_{\nu}^+,t_{\nu}^{l_1}))_{r_{\nu}^{l_1}}$, by $i_{\nu}^{l_1}\in [0,C_2(h(t_k)+g(t_k))]$ and the definition of $G_{\nu}^+$, for each $t'>C_2(h(t_k)+g(t_k))$, we have
$$
\frac{\int_0^{t'} \chi_{B_{\nu}}(f_{s}(y_1))ds}{t'}<\frac{2}{5},
$$
which implies that
$$
\frac{\int_0^{t'} \chi_{G_{\nu}}(f_{s}(y_1))ds}{t'}\geq \frac{3}{5}
$$
for all such $t'$. Since $d_{r_{\nu}^{l_1}}(y_1,f_{r_{\mu}^{l_1}+\tau_1}(z_1))\leq \eta$, we know
\begin{equation} \label{eqgoodnularge}
    \frac{\int_0^{t'} \chi_{B_{\eta}(G_{\nu})}(f_{s}(f_{r_{\mu}^{l_1}+\tau_1}(z_1)))ds}{t'}\geq \frac{3}{5}.
\end{equation}
Now we turn to the trajectory of $x_2$. Since $x_2\in (f_{i_{\mu}^{l_2},j_{\mu}^{l_2}}(f_{-t_{\mu}^{l_2}}(G_{\mu}^-),t_{\mu}^{l_2}))_{r_{\mu}^{l_2}}$, we know $f_{t_{\mu}^{l_2}-i_{\mu}^{l_2}}(x_2)=f_{r_{\mu}^{l_2}+j_{\mu}^{l_2}}(x_2)\in G_{\mu}^-$. As in the case for $y_1$, by $j_{\mu}^{l_2}\in [0,C_2(h(t_k)+g(t_k))]$ and the definition of $G_{\mu}^-$, for each $t'>C_2(h(t_k)+g(t_k))$, we have
$$
\frac{\int_0^{t'} \chi_{B_{\mu}}(f_{-s}(f_{r_{\mu}^{l_2}}(x_2)))ds}{t'}<\frac{2}{5},
$$
which implies that
$$
\frac{\int_0^{t'} \chi_{G_{\mu}}(f_{-s}(f_{r_{\mu}^{l_2}}(x_2)))ds}{t'}\geq \frac{3}{5},
$$
for all such $t'$. By $d_{r_{\mu}^{l_2}(x_2,z_2)}\leq \eta$, we have
\begin{equation} \label{eqgoodmularge}
    \frac{\int_0^{t'} \chi_{B_{\eta}(G_{\mu})}(f_{-s}(f_{r_{\mu}^{l_2}}(z_2)))ds}{t'}\geq \frac{3}{5}.
\end{equation}
We will combine \eqref{eqgoodnularge} and \eqref{eqgoodmularge} by concentrating on the time interval $I_{z_1,z_2}:=[r_{\mu}^{l_1}+\tau_1,r_{\mu}^{l_2}]$. The length of $I_{z_1,z_2}$ equals 
$$L_{z_1,z_2}=r_{\mu}^{l_2}-r_{\mu}^{l_1}-\tau_1>C_2(h(t_k)+g(t_k)).$$ 
Therefore, \eqref{eqgoodnularge} and \eqref{eqgoodmularge} respectively imply that
$$
\int_{r_{\mu}^{l_1}+\tau_1}^{r_{\mu}^{l_2}}\chi_{B_{\eta}(G_{\nu})}(f_s(z_1))ds\geq \frac{3L_{z_1,z_2}}{5},
$$
and
$$
\int_{r_{\mu}^{l_1}+\tau_1}^{r_{\mu}^{l_2}}\chi_{B_{\eta}(G_{\mu})}(f_s(z_2))ds\geq \frac{3L_{z_1,z_2}}{5}.
$$
Since $d(B_{\eta}(G_{\nu}),B_{\eta}(G_{\mu}))\geq 5\eta$ by our choice on $\eta$, we must have
$$
d_{L_{z_1,z_2}}(f_{r_{\mu}^{l_1}+\tau_1}(z_1),f_{r_{\mu}^{l_1}+\tau_1}(z_2))\geq 5\eta,
$$
which concludes the claim.
\end{proof}

\section{Applications}

\subsection{Suspension flow} 
\subsubsection{Preliminaries}
We start with a brief review on the notion of suspension flow. Let $f:X\to X$ be a homeomorphism of a  compact metric space $(X,d)$ and $r\in C(X,\mathbb{R}^+)$. Consider the space
$$
Y=\{(x,s)\in X\times \mathbb{R}:0\leq s\leq r(x)\},
$$
and let $X_r$ be the set obtained from $Y$ by identifying $(x,r(x))$ with $(f(x),0)$ for each $x\in X$. Such $r$ is often called a \emph{roof function}. 
\begin{definition}
    The suspension flow over $(X,f)$ with roof function $r$ is the flow $F=\{f_t\}_{t\in \mathbb{R}}$ with $f_t:X_{r}\to X_r$ satisfying
    $$
f_t(x,s)=(x,s+t).
    $$
\end{definition}
We often extend $r$ to a function on $X_r$ by letting $r(x'):=\min\{t>0:f_t(x')\in X\times \{0\}\}$. To formulate the metric on $X_r$, we first introduce so-called \emph{Bowen-Walters distance} on constant-time suspension with $r=1$, denoted by $d_{X_1}$. Given $x,y\in X$ and $t\in [0,1]$, for the \emph{horizontal segment} $[(x,t),(y,t)]$ in $X_1$, let the corresponding length be
$$
d_{1,h}((x,t),(y,t)):=(1-t)d(x,y)+td(f(x),f(y)).
$$
Meanwhile, for $(x,t),(x',t')$ being in the same $F$-orbit, let the length of the corresponding \emph{vertical segment} $[(x,t),(x',t')]$ be
$$
d_{1,v}((x,t),(x',t'))=\inf\{|s|:f_{s}(x,t)=(x',t')\}.
$$
Then for any $(x,t),(y,s)\in X_1$, their \emph{Bowen-Walters distance} $d_{BW}((x,t),(y,s))$ is defined as the infimum of lengths of all paths connecting $(x,t),(y,s)$ that are composed of horizontal and vertical segments. For general $r$, given any $(x,t),(y,s)\in X_r$, the metric $d_{X_r}$ on $X_r$ is given by
$$
d_{X_r}((x,t),(y,s)):=d_{BW}((x,t/r(x)),(y,s/r(y))).
$$
For $f$ bi-Lipschitz continuous and $r$ Lipschitz continuous, it is well-known that $d_{X_r}$ can be measured in another manner; see \cite[Proposition 2.1]{Ba13}. Throughout this section, we \emph{only require $f$ to be Lipschitz}; that is, we do not need $f^{-1}$ or $r$ to be Lipschitz. Given any $(x,t),(y,s)\in X_r$, let 
$$
d_{X_r}^+((x,t),(y,s)):=\min\left\{
\begin{array}{cc}
     &  d(x,y)+|t/r(x)-s/r(y)| \\
     &  d(f(x),y)+1-t/r(x)+s/r(y)  \\
     &  d(x,f(y))+1-s/r(y)+t/r(x)
\end{array}
\right\},
$$
and
$$
d_{X_r}^-((x,t),(y,s)):=\min\left\{
\begin{array}{cc}
     &  d(f(x),f(y))+|t/r(x)-s/r(y)| \\
     &  d(f^2(x),f(y))+1-t/r(x)+s/r(y)  \\
     &  d(f(x),f^2(y))+1-s/r(y)+t/r(x)
\end{array}
\right\},
$$
By running the arguments in the proof for \cite[Proposition 2.1]{Ba13}, yet not using Lipschitz property of $r$ in expanding $|t/r(x)-s/r(y)|$, nor Lipschitz property of $f^{-1}$ in justifying the lower bound for $d_{X_r}$, we have the following result relating $d_{X_r}$ and $d_{X_r}^+,d_{X_r}^-$. Details are left to interested readers.
\begin{lemma} \label{lem:distancecomparison}
    If $f$ is Lipschitz, then there exists a constant $c=c(f,r)$ such that
\begin{equation}
    c^{-1}d_{X_r}^-(p,q)\leq d_{X_r}(p,q)\leq c d_{X_r}^+(p,q) \text{ for all } p,q\in X_r
\end{equation}
\end{lemma}

We will consider a suspension flow induced by $(X,f,r)$, with the base dynamical system $(X,f)$ satisfying expansiveness, weak controlled specification property at all scales with gap function $h$ having slow growth rate. As in the general study of suspension flow, the roof function $r$ is always expected to satisfy some regularity requirement, which in our case is stated as follows. Given any $\delta>0$, $n,m\in \mathbb{N}^+$, write
$$
\text{Var}_m^{\pm}(r,n,\delta):=\sup_{x\in X,y\in B^f_{[-m+1,n+m-1]}(x,\delta)}|S_nr(x)-S_nr(y)|
$$
where $B^f_{[-n_1,n_2]}(x,\delta):=\{y\in X: d(f^i(x),f^i(y))<\delta,\ \forall i\in [-n_1,n_2]\cap \mathbb{Z}\}$ for all $n_1,n_2\in \mathbb{N}$.
\begin{definition} \label{DefweakWalters}
Given any $\delta>0$, we say the roof function $r\in C(X,\mathbb{R}^+)$ satisfies \emph{weak Walters condition at scale} $\delta$ if there exists a monotonically increasing function $k:\mathbb{N}^+\to \mathbb{N}^+$ satisfying $\lim_{n\to \infty}\frac{k(n)}{\log n}=0$, such that
$$
 \lim_{n\to \infty}\text{Var}_{k(n)}^{\pm}(r,n,\delta)=0.
$$
We also say $r$ satisfies \emph{weak Walters condition} if it does at some scale. 
\end{definition}
\begin{remark}
    Recall that $r$ is said to satisfy \emph{Walters condition} at scale $\delta$ if 
    $$\lim_{m\to \infty}\sup_{n\in \mathbb{N}^+}\{\text{Var}^{\pm}_{m}(r,n,\delta)\}=0.
    $$ 
    It is not hard to observe that Walters condition implies weak Walters condition (at the same scale). It is also clear that if $r$ satisfies weak Walters condition at scale $\delta$, then it does at all smaller scales.
\end{remark}

\subsubsection{Proof of Theorem \ref{thm:suspension}} This section is devoted to the proof of Theorem \ref{thm:suspension}. The main strategy lies in lifting expansiveness and specification property from underlying discrete system $(X,f)$ to the continuous system $(X_r,F)$. Indeed, by Proposition \ref{expan} below, the flow $(X_r,F)$ is entropy-expansive at any small scale. Meanwhile, Proposition \ref{lem:suspensionspecification} indicates that $(X_r,F)$ satisfies weak controlled specification at any scale $\eta$ with a gap function $h_{r,\eta}$ satisfying $\liminf_{t\to \infty}\frac{h_{r,\eta}(t)}{\log t}=0$. Theorem \ref{thm:suspension} is therefore concluded by an immediate application of Theorem \ref{thm:mme}.

By a uniform scaling, from now on, we assume that $r(X)\subset [1,L]$ for some constant $L\geq 1$. Let us start with expansiveness by first choosing a few constants. Fix a constant $\eps$ at which both expansiveness for $(X,f)$ and weak Walters condition for $r$ hold. Fix $\delta_0\ll \eps$, and choose $\eta_0,\eps_0\in (0,\delta_0)$ iteratively such that
$$
d(x,y)<\eta_0\implies |r(x)-r(y)|<\delta_0 
$$
and
$$
d(x,y)<\eps_0 \implies \max\{d(f^{-1}(x),f^{-1}(y)),d(f^{-2}(x),f^{-2}(y))\}<\eta_0.
$$
We also write $\eps_0':=\eps_0/cL$, where $c=c(f,r)$ is as in Lemma \ref{lem:distancecomparison}. For any $p\in X_r$ and $\eta\in (0,\eps)$, recall that
$$
\Gamma_{\eta}^F(p):=\{q\in X_r:d_{X_r}(f_t(p),f_t(q))<\eta \text{ for all }t\in \mathbb{R}\}.
$$
\begin{lemma} \label{lem:non-expansiveset}
    For any $\eps'\in (0,\eps_0')$ and $p\in X_r$, we have $\Gamma_{\eps'}^F(p)\subset f_{[-\delta_0,\delta_0]}(p)$. 
\end{lemma}
\begin{proof}
    Fix any $\eps'\in (0,\eps_0')$. By flowing in the flow direction for at most $L$ time, we may assume $p=(x,1/3)$. For any $q=(y,s)\in\Gamma_{\eps'}^F(p)$, since $\min\{1-1/3r(x),1/3r(x)\}\geq 1/3L>\eps'$, by applying Lemma \ref{lem:distancecomparison}, we have
\begin{equation} \label{eqdistancepqcontrol}
    d(f(x),f(y))+|1/3r(x)-s/r(y)|<c\eps'<\eps_0,
\end{equation}
    which implies that $\max\{d(x,y),d(f^{-1}(x),f^{-1}(y))\}<\eta_0$, and therefore
\begin{equation} \label{eqdistancercontrol}
    \max\{|r(x)-r(y)|,|r(f^{-1}(x))-r(f^{-1}(y))|\}<\delta_0.
\end{equation}
It then follows from \eqref{eqdistancepqcontrol}, \eqref{eqdistancercontrol} and our choice on $\eps_0'$ that
$$
s\in (1/3-4\delta_0/3,1/3+4\delta_0/3),
$$
which together with \eqref{eqdistancercontrol} implies  
$$
r(x)+s\in (r(y),r(y)+1) \quad \text{ and } \quad s-r(f^{-1}(x))\in (-r(f^{-1}(y)),-r(f^{-1}(y))+1).
$$
Therefore, we know
$$
f_{r(x)}(q)=(f(y),r(x)+s-r(y)), \quad f_{-r(f^{-1}(x))}(q)=(f^{-1}(y),s+r(f^{-1}(y))-r(f^{-1}(x))).
$$
Writing $s_1:=r(x)+s-r(y)$, $s_{-1}:=s+r(f^{-1}(y))-r(f^{-1}(x))$. Observe that $f_{r(x)}(p)=(f(x),1/3)$, $f_{r(f^{-1}(x))}(p)=(f^{-1}(x),1/3)$. Then $d_{X_r}((f(x),1/3),(f(y),s_1))<\eps'$, which by Lemma \ref{lem:distancecomparison} implies that $d(f^2(x),f^2(y))+|1/3r(f(x))-s_1/r(f(y))|<c\eps'$, as an analog of \eqref{eqdistancepqcontrol}. By repeating the argument from above iteratively in the forward direction, we have
$$
d(f^i(x),f^i(y))<\eps_0 \quad \text{ for all }i\in \mathbb{N}^+.
$$
Similarly, a symmetric argument initiating from $d_{X_r}((f^{-1}(x),1/3),(f^{-1}(y),s_{-1}))<\eps'$ in the backward direction implies that
$$
d(f^{-i}(x),f^{-i}(y))<\eps' \quad \text{ for all }i\in \mathbb{N}.
$$
Consequently, we know $y\in \Gamma_{\eps_0}^f(x)=\{x\}$, which together with \eqref{eqdistancercontrol} implies that
$q=(y,s)\in f_{[-\delta_0,\delta_0]}(p)$, thus concludes the lemma.
\end{proof}
As an immediate consequence of the lemma above, we have
\begin{proposition}\label{expan}
    The flow $(X_r,F)$ is entropy-expansive at any scale less than $\eps_0'$.
\end{proposition}
 
The establishment of our desired specification property for $(X_r,F)$ relies on a combination of specification for $(X,f)$ and regularity of $r$. Let $k$ be the function in the definition of weak Walters condition for $r$ at scale $\eps$, which satisfies $\lim_{n\to \infty}\frac{k(n)}{\log n}=0$. 
\begin{lemma} \label{lem:partialsumrvariation}
    For any $\eta>0$, there exists $N=N(\eta)\in \mathbb{N}$ such that for any $x\in X$, $n\in \mathbb{N}^+$, $y\in B_{[-N-k(n)+1,n+N+k(n)-1]}^f(x,\eps)$, we have
    $$
|S_nr(x)-S_nr(y)|<\eta.
    $$
\end{lemma}
\begin{proof}
    Fix any $\eta>0$ and $x\in X$. Since $r$ satisfies weak Walters condition, there exists $N_1=N_1(\eta)$ such that 
    \begin{equation} \label{eq>N1}
        \text{Var}_{k(n)}(r,n,\eps)<\eta \quad \text{ for all }n>N_1.
    \end{equation}
    For $n\leq N_1$, let $\eta'>0$ be a constant such that $d(x_1,x_2)<\eta'\implies |r(x_1)-r(x_2)|<\eta/N_1$. Let $N\in \mathbb{N}^+$ be such that
    $$
x_2\in B^f_{[-N+1,N-1]}(x_1,\eps) \implies d(x_1,x_2)<\eta'.
    $$
    Then for any $y\in B_{[-N-k(n)+1,n+N+k(n)-1]}^f(x,\eps)$, we know $d(f^i(x),f^i(y))<\eta'$ for all $i\in [0,n-1]$, which gives $|r(f^i(x))-r(f^i(y))|<\eta/N_1$ for all such $i$. Consequently, we know 
    \begin{equation} \label{eq<N1}
        |S_nr(x)-S_nr(y)|<\eta \quad \text{ for all such }y \text{ with }n\leq N_1.
    \end{equation}
    The lemma now follows from a combination of \eqref{eq>N1} and \eqref{eq<N1}.
\end{proof}

\begin{lemma} \label{lem:distanceXrclose}
    For any $\eta>0$, there exists $\eps_1=\eps_1(\eta)\in (0,\eps)$ such that for any $x\in X$, $n\in \mathbb{N}^+$, if $y\in B^f_{[-N-k(n)+1,N+k(n)+n-1]}(x,\eps_1)$, where $N=N(\eta/5c)$ with $c=c(f,r)$ being as in Lemma \ref{lem:distancecomparison}, then writing $p=(x,0)$, $q=(y,0)$, we have
    $$
d_{X_r}(f_t(p),f_t(q))<\eta \quad \text{ for all } t\in [0,S_nr(x)].
    $$
\end{lemma}
By expansiveness of $(X,f)$ at scale $\eps$ and compactness of $X$, as an immediate corollary of Lemma \ref{lem:distanceXrclose}, we have
\begin{corollary} \label{coro:distanceXrclose}
    For any $\eta>0$, there exists $N_1=N_1(\eta)\in \mathbb{N}^+$ such that for any $x\in X$, $n\in \mathbb{N}^+$, if $y\in B^f_{[-N-N_1-k(n)+1,N+N_1+k(n)+n-1]}(x,\eps)$, then
    $$
d_{X_r}(f_t(p),f_t(q))<\eta \quad \text{ for all } t\in [0,S_nr(x)],
    $$
    where $p=(x,0)$, $q=(y,0)$.
\end{corollary}
\begin{proof}[Proof of Lemma \ref{lem:distanceXrclose}]
    Fix any $\eta>0$, and write $\eta':=\eta/5c$. Let $\eps_1:=\min \{\eta',\eps/2\}$. Let $x,n,y$ be as in the lemma. We will prove the lemma by discussing $t\in [0,S_n r(x)]$ by cases. 

    \emph{Case I:} when $t\in \bigcup_{m=0}^{n-1}[S_mr(x)+2\eta',S_{m+1}r(x)-2\eta']$, then $t\in [S_mr(x)+2\eta',S_{m+1}r(x)-2\eta']$ for some $m\in [0,n-1]$. Since $|S_mr(x)-S_mr(y)|<\eta'$ by Lemma \ref{lem:partialsumrvariation}, we know $f_t(p)=(f^m(x),s_t)$, $f_t(q)=(f^m(y),s_t')$, where $s_t\in [2\eta',r(f^m(x))-2\eta']$, $s_t'\in [s_t-\eta',s_t+\eta']$. Since
\begin{equation} \label{eq:stst'difference}
    |s_t/r(f^m(x))-s_t'/r(f^m(y))|\leq |r(f^m(x))-r(f^m(y))|+|s_t-s_t'|\leq 2\eta',
\end{equation}
    it follows from Lemma \ref{lem:distancecomparison} that
    \begin{equation} \label{eq:CaseIdxrftpftq}
        d_{X_r}(f_t(p),f_t(q))\leq c(d(f^m(x),f^m(y))+|s_t/r(f^m(x))-s_t'/r(f^m(y))|)\leq 3c\eta'<\eta
    \end{equation}
    \emph{Case II:} when $t\in [0,S_nr(x)]\setminus (\bigcup_{m=0}^{n-1}[S_mr(x)+2\eta',S_{m+1}r(x)-2\eta'])$, either $t\in (S_mr(x),S_mr(x)+2\eta')$ for some $m\in [0,n-1]$, or $t\in (S_mr(x)-2\eta',S_mr(x))$ for some $m\in [1,n]$. We assume the first, as the second can be derived via a symmetric argument. Write $f_t(p):=(f^m(x),s_t)$. We still have $|S_mr(x)-S_mr(y)|<\eta'$ by Lemma \ref{lem:partialsumrvariation}, yet the first coordinate of $f_t(q)$ have two possiblities in this case.
    \begin{enumerate}
        \item If $f_t(q)=(f^{m-1}(y),s_t')$, since $|S_{m-1}r(x)-S_{m-1}r(y)|<\eta'$ by Lemma \ref{lem:partialsumrvariation}, together with $t\in [S_mr(x),S_mr(x)+2\eta]$ and $|S_mr(x)-S_mr(y)|<\eta'$, we know $s_t'=t-S_{m-1}r(y)\in [r(f^{m-1}(y))-2\eta',r(f^{m-1}(y))]$. Then it follows from Lemma \ref{lem:distancecomparison} that
        \begin{equation} \label{eq:CaseIIdxrftpftq1}
           d_{X_r}(f_t(p),f_t(q))\leq c(d(f^m(x),f^m(y))+1-s_t'/r(f^{m-1}(y))+s_t/r(f^m(x)))<5c\eta'=\eta. 
        \end{equation}
    
    \item If $f_t(q)=(f^m(y),s_t')$, as $|s_t-s_t'|=|S_mr(x)-S_mr(y)|<\eta'$, again by \eqref{eq:stst'difference} and Lemma \ref{lem:distancecomparison} we know
    \begin{equation} \label{eq:CaseIIdxrftpftq2}
        d_{X_r}(f_t(p),f_t(q))\leq c(d(f^m(x),f^m(y))+|s_t/r(f^m(x))-s_t'/r(f^m(y))|)<3c\eta'<\eta.
    \end{equation}
    \end{enumerate}
    The proof of Lemma \ref{lem:distanceXrclose} is now concluded by combining \eqref{eq:CaseIdxrftpftq}, \eqref{eq:CaseIIdxrftpftq1}, and \eqref{eq:CaseIIdxrftpftq2}.
\end{proof}

We are now ready with the ingredients for the specification property of $(X_r,F)$, which is stated as follows.
\begin{proposition} \label{lem:suspensionspecification}
    For any $\eta>0$, there exists a monotonically increasing function $h_r=h_{r,\eta}:\mathbb{R}^+\to \mathbb{R}^+$ such that $(X_r,F)$ satisfies weak controlled specification at scale $\eta$ with gap function $h_r$, and $\liminf_{t\to \infty}\frac{h(t)}{\log t}=0$.
\end{proposition}
\begin{proof}
    Fix $\eta>0$. Let $\{(p_i,s_i)\}_{i=1}^n$ be a sequence of orbit segments in $X_r$. We first consider the case where each $(p_i,s_i)$ is in the form of $((x_i,0),S_{n_i}r(x_i))$ with $x_i\in X$ and $n_i\in \mathbb{N}^+$ for all $i$. Let $p_i':=(x_i',0)$ where $x_i':=f^{-(N+N_1+k(n_i)-1)}(x_i)$ with $N_1=N_1(\eta)$, and $s_i':=S_{n_i'}r(x_i')$, where $n_i':=2N+2N_1+2k(n_i)+n_i-1$. By applying weak controlled specification of $(X,f)$ at scale $\eps$ with gap function $h$ to $\{(x_i',n_i')\}_{i=1}^n$, there exist a sequence of non-negative integers $\{\tau_i'\}_{i=1}^{n-1}$ with $\tau_i'\leq \max\{h(n_i'),h(n'_{i+1})\}$, and a point $y'\in X$ such that 
    $$
d_{n_i'}(f^{\sum_{j=1}^{i-1}\tau_j'+n'_{j}}(y),x_i')<\eps.
    $$
    Let $y_i':=f^{\sum_{j=1}^{i-1}\tau_j'+n'_{j}}(y)$, and $y_i:=f^{N+N_1+k(n_i)-1}(y_i')$. Also write $q_i:=(y_i,0)$ and $q_i':=(y_i',0)$. By applying Corollary \ref{coro:distanceXrclose} to $\eta$, $x_i,n_i,y_i$, we know
    \begin{equation}
        d_{X_r}(f_t(p_i),f_t(q_i))<\eta \quad \text{ for all }t\in [0, s_i].
    \end{equation}
    Let $q:=q_1=(y_1,0)$. Notice that $f_{\tau_i+s_i}(q_i)=q_{i+1}$, where 
    $$
\tau_i=S_{2N+2N_1+k(n_i)+k(n_{i+1})+\tau_i'}r(y_i).
    $$
    Then $\tau_i$ is exactly the transition time $f^{s_i+\sum_{j=1}^{i-1}(s_j+\tau_j)}(q)$ spent on traveling to the $(s_{i+1},\eta)$-neighborhood of $p_{i+1}$. For each $t>0$, let
    $$
h_r(t):=2L(N+N_1+h(\lceil t \rceil)+k(\lceil t \rceil)),
    $$
    where $\lceil t \rceil$ denotes the smallest integer greater than or equal to $t$. By our choice on $\tau_i'$ and $r(X)\subset [1,L]$, it is clear that
    $$
\tau_i\leq 2L(N+N_1+\max\{k(n_i)+h(n_i),k(n_{i+1})+h(n_{i+1})\})\leq \max\{h_r(s_i),h_r(s_{i+1})\}.
    $$
    This concludes the case where each $(p_i,s_i)$ takes a special form. The general case follows immediately by stretching each orbit for at most $L$ time in both directions, and invoking the result from above. The corresponding gap function $h_r$ in the general case is therefore
    $$
h_r(t):=2L(N+N_1+h(\lceil t+2L \rceil)+k(\lceil t+2L \rceil)),
    $$
    which clearly satisfies $\liminf_{t\to \infty}\frac{h(t)}{\log t}=0$ by the definition of $h$ and $k$.
\end{proof}

\subsection{Equilibrium states of frame flows in nonpositive curvature} 

\subsubsection{Frame flows in nonpositive curvature}
Suppose that $(M,g)$ is a $C^{\infty}$ closed $n$-dimensional Riemannian manifold,
where $g$ is a Riemannian metric of nonpositive (sectional) curvature. Let $\bar \pi: SM\to M$ be the unit tangent bundle over $M$. For each $v\in S_pM$,
we always denote by $\gamma_{v}: \RR\to M$ the unique geodesic on $M$ satisfying the initial conditions $\gamma_v(0)=p$ and $\dot \gamma_v(0)=v$. The geodesic flow $(g^{t})_{t\in\mathbb{R}}$ (generated by the Riemannian metric $g$) on $SM$ is defined as:
\[
g^{t}: SM \rightarrow SM, \qquad (p,v) \mapsto
(\gamma_{v}(t),\dot \gamma_{v}(t)),\ \ \ \ \forall\ t\in \RR .
\]

A vector field $J(t)$ along a geodesic $\gamma:\RR\to M$ is called a \emph{Jacobi field} if it satisfies the \emph{Jacobi equation}:
\[J''+R(J, \dot \gamma)\dot \gamma=0\]
where $R$ is the Riemannian curvature tensor and\ $'$\ denotes the covariant derivative along $\gamma$.

A Jacobi field $J(t)$ along a geodesic $\gamma(t)$ is called \emph{parallel} if $J'(t)=0$ for all $t\in \RR$. The notion of \emph{rank} is defined as follows.
\begin{definition}
For each $v \in SM$, we define \text{rank}($v$) to be the dimension of the vector space of parallel Jacobi fields along the geodesic $\gamma_{v}$, and 
\[\text{rank}(M):=\min\{\text{rank}(v): v \in SM\}.\] 
For a geodesic $\gamma$ we define $\text{rank}(\gamma):=\text{rank}(\dot \gamma(t)), \forall\ t\in \mathbb{R}.$
\end{definition}

Let $M$ be a rank one closed manifold of nonpositive curvature. Then $SM$ splits into two invariant subsets under the geodesic flow: the regular set $\text{Reg}:= \{v\in SM: \text{rank}(v)=1\}$, and the singular set $\text{Sing}:= SM \setminus \text{Reg}$.

Let $X$ be the universal cover of $M$. We call two geodesics $\gamma_{1}$ and $\gamma_{2}$ on $X$ \emph{positively asymptotic or asymptotic} if there is a positive number $C > 0$ such that $d(\gamma_{1}(t),\gamma_{2}(t)) \leq C, ~~\forall~ t \geq 0.$
The relation of asymptoticity is an equivalence relation between geodesics on $X$. The class of geodesics that are asymptotic to a given geodesic $\gamma_v/\gamma_{-v}$ is denoted by $\gamma_v(+\infty)$/$\gamma_v(-\infty)$ or $v^+/v^-$ respectively. We call them \emph{points at infinity}. Obviously, $\gamma_{v}(-\infty)=\gamma_{-v}(+\infty)$. We call the set $\partial X$ of all points at infinity the \emph{boundary at infinity}.

The frame bundle is defined as
\begin{equation*}
\begin{aligned}
FM:=\{&(x, v_0, v_1, \cdots, v_{n-1}): x\in M, v_i\in T_xM, \\
&\{v_0, \cdots, v_{n-1}\} \text{\ is a positively oriented orthonormal frame at\ } x\}.
\end{aligned}
\end{equation*}
The frame flow $F^t:FM\to FM$ is defined by
$$F^t(x, v_0, v_1, \cdots, v_{n-1}):= (g^t(v_0), P_\gamma^t(v_1),\cdots, P_\gamma^t(v_{n-1}))$$
where $P_\gamma^t$ is the parallel transport along the geodesic $\gamma_{v_0}(t)$.
There is a natural fiber bundle $\pi: FM\to SM$ that takes a frame to its first vector, and then each fiber can be identified with $SO(n-1)$. We have $\pi \circ F^t=g^t\circ \pi$ and $F^t$ acts isometrically along the fibers.

The study of frame flows on manifolds of negative curvature has been developed in 1970s by Brin, Gromov \cite{Brin75, Brin80, BG} etc. In this case, the frame flow $F^t$ is a partially hyperbolic flow, that is, there exists an $F^t$-invariant splitting $TFM = E_F^s \oplus E_F^c \oplus E_F^u$, where $\dim E_F^c = 1+\dim SO(n-1)$ and $F^t$ acts isometrically on the center bundle. The following are the cases when $F^t$ is known to be ergodic with respect to the normalized volume measure on $FM$, and hence to be topologically transitive.
\begin{proposition}(\cite[Theorem 0.2]{BP03})
Let $F^t$ be the frame flow on an $n$-dimensional compact
smooth Riemannian manifold with sectional curvature between $-\Lambda^2$ and $-\lambda^2$ for some
$\Lambda, \lambda>0$. Then in each of the following cases the flow and its time-one map are ergodic:
\begin{itemize}
  \item if the curvature is constant,
  \item for a set of metrics of negative curvature which is open and dense in the $C^3$ topology,
  \item if $n$ is odd and $n\neq 7$,
  \item if $n$ is even, $n\neq 8$, and $\lambda/\Lambda>0.93$,
  \item if $n=7$ or $8$ and $\lambda/\Lambda>0.99023...$.
\end{itemize}
\end{proposition}

Spatzier and Visscher \cite[Theorem 1]{SV} obtained the following result for manifolds of negative curvature. Similar results are also obtained by Climenhaga, Pesin and Zelerowicz \cite[Theorem 5.8 and Corollary 5.10]{CPZ}, Carrasco and Rodriguez-Hertz \cite[Theorem A]{CRH}, etc.
\begin{theorem}(cf. \cite[Theorem 1]{SV})
Let $M$ be a closed, oriented, negatively curved $n$-manifold, with $n$ odd and not equal to $7$. For any H\"{o}lder continuous potential $\varphi: FM\to \mathbb{R}$ that is constant on the fibers of the bundle $FM\to SM$, there is a unique equilibrium state for $(F^t, \varphi)$. It is ergodic and has full support.
\end{theorem}

Recently, for geodesic flows on rank one closed manifolds of nonpositive curvature, Burns, Climenhaga, Fisher and Thompson \cite{BCFT} proved the uniqueness of equilibrium states for certain potentials:
\begin{theorem}(cf. \cite[Theorem A]{BCFT})
Let $g^t$ be the geodesic flow over a closed rank 1 manifold $M$ and let $\varphi :SM\to \RR$ be H\"{o}lder continuous or be $\varphi=q\phi_u$ where $q\in \RR$ and
$\phi_u(v):=-\lim_{t\to 0}\frac{1}{t}\log \det (Dg^t|_{E^u(v)}), v\in SM.$
If $P(Sing,\varphi) <P(\varphi)$, then $\varphi$ has a unique equilibrium state $\mu$. This equilibrium state is hyperbolic, fully supported, and is the weak$^{*}$ limit of weighted regular closed geodesics.
\end{theorem}

We consider the uniqueness of equilibrium states for frame flows on rank one closed manifolds of nonpositive curvature. Additionally, our target manifold has their pointwise collection of sectional curvatures being uniformly comparable in the following sense. For $v,w\in T_xM$, we denote by $K(v,w)$ the sectional curvature of the plane spanned by $v$ and $w$.
\begin{definition}
    We say $M$ is \emph{bunched} if there exists a constant $C_B>1$ such that for every $x\in M$ and $\{v_i\}_{i=1}^4\subset T_xM$, we have 
    $$
K(v_1,v_2)\in [C_B^{-1}K(v_3,v_4),C_BK(v_3,v_4)]
    $$
    whenever the sectional curvature is defined. 
\end{definition}
In particular, surfaces of nonpositive curvature and closed manifolds of negative curvature are all bunched. 
In the rest of this section, we always assume that $M$ is a rank one closed Riemannian manifold of bunched nonpositive curvature, whose frame flow $F^t: FM\to FM$ is topologically transitive. Our aim is to prove Theorem \ref{frameflow}, based on the application of Theorem \ref{thm:weprove}.

\subsubsection{Decomposition for the geodesic flow}
For basic notions on the geometry of geodesic flows, we refer to \cite[Section 2.4]{BCFT}. In particular, we have $g^t$-invariant subbundles $E^s$ and $E^u$ of $TSM$, and they are integrable to $g^t$-invariant foliations $W_g^s$ and $W_g^u$ respectively. For $v\in SM$, we call $W_g^{s/u}(v)$ the stable/unstable manifolds of the geodesic flow through $v$. We denote by $W_g^{s/u}(v,\delta)$ the ball of radius $\delta>0$ centered at $v$ with respect to the intrinsic metrics $d^{s/u}$ on $W_g^{s/u}(v)$.  
The notations for weak stable/unstable foliations $W_g^{cs/cu}$ are defined in a similar way.

For $v\in SM$, let $H^{s/u}(v)$ be the stable/unstable horospheres associated to $v$. Therefore, we have $H^{s/u}(v)=\bar\pi W_g^{s/u}(v)$. Recall that $\UUU_v^s: T_{\bar\pi v}H^s\to  T_{\bar\pi v}H^s$ is the symmetric linear operator associated to the stable horosphere $H^s$, and similarly for $\UUU_v^s$. Let $\lambda^u(v)$ be the minimal eigenvalue of $\UUU_v^u$ and let $\lambda^s(v)=-\lambda^u(-v)$. Then we define $\lambda(v)=\min \{\lambda^u(v), \lambda^s(v)\}.$ It is clear that $\lambda: SM\to \RR$ is a continuous function. 

Let $\eta>0$. Define 
\begin{equation*}
    \begin{aligned}
\GGG(\eta)&:=\{(v,t): \int_0^\tau \lambda(g^sv)ds \ge \eta\tau, \int_0^\tau \lambda(g^{-s}g^tv)ds \ge \eta\tau, \forall \tau \in [0,t]\},\\
\BBB(\eta)&:=\{(v,t): \int_0^t \lambda(g^sv)ds < \eta t\},\\
\text{Reg}(\eta)&:=\{v\in SM: \lambda(v)\ge \eta\},\\
\CCC(\eta)&:=\{(v,t)\in SM\times \RR^+: v\in \text{Reg}(\eta), \ g^tv\in \text{Reg}(\eta)\}.
    \end{aligned}
\end{equation*}
\begin{definition}\label{decom}
Given $(v,t)\in SM\times\RR^+$, take $p=p(v,t)$ to be the largest time such that $(v,p)\in \BBB(\eta)$, and $s=s(v,t)$ be the largest time in $[0, t-p]$ such that $(g^{t-s}v, s)\in \BBB(\eta)$. Then it follows that $(g^pv, g)\in \GGG(\eta)$ where $g=t-p-s$. Thus the triple $(\BBB(\eta), \GGG(\eta), \BBB(\eta))$ equipped with the functions $(p,g,s)$ determines a decomposition for $SM\times \RR^+$.
\end{definition}

Given $\eta>0$, let $\delta=\delta(\eta)>0$ be small enough so that for any $v,w\in SM$, 
\begin{equation}\label{smalldelta}
d(v,w)< \delta e^\Lambda \Rightarrow |\lambda(v)-\lambda(w)|\le \eta/2.    
\end{equation}
Here $\Lambda$ is the maximum eigenvalue of $\UUU^u(v)$ taken over all $v\in SM$. 
\begin{lemma}(\cite[Corollar 3.11]{BCFT})\label{stablegeo}
Given $\eta>0$, let $\delta=\delta(\eta)$ be as in \eqref{smalldelta} and $(v,T)\in \GGG(\eta)$, then every $w\in B_T (v, \delta)$ has $(w,T)\in \GGG(\eta/2)$. Moreover, for every $w,w'\in W_g^s(v,\delta)$ and $0\le t\le T$ we have
$$d^s(g^tw, g^tw')\le d^s(w,w')e^{-\frac{\eta}{2}t},$$
and  for every $w,w'\in g^{-T}W_g^u(g^Tv,\delta)$ and $0\le t\le T$, we have
$$d^u(g^tw, g^tw')\le d^u(g^Tw,g^Tw')e^{-\frac{\eta}{2}(T-t)}.$$
\end{lemma}

\begin{definition}\label{LPSdef}
The foliations $W_g^u$, $W_g^{cs}$ have local product structure (LPS) with constant $\kappa\ge 1$ in a $\delta$-neighborhood of $v\in SM$ if for every $\epsilon\in (0,\delta]$ and all $w_1, w_2\in B(v, \epsilon)$, the intersection $W_g^u(w_1, \kappa \epsilon) \cap W_g^{cs}(w_2, \kappa \epsilon)$ contains a single point, which we denote by the Smale bracket $[w_1, w_2]_1$, and if moreover we have
\begin{equation*}
    \begin{aligned}
d_g^u(w_1, [w_1, w_2]_1) &\le \kappa d(w_1, w_2),\\
d_g^{cs}(w_2, [w_1, w_2]_1) &\le \kappa d(w_1, w_2).
\end{aligned}
\end{equation*}
We can define analogously local product structure (LPS) for $W_g^s$, $W_g^{cu}$ with Smale bracket $[\cdot, \cdot]_2$.
\end{definition}

\begin{lemma}(\cite[Lemma 4.4]{BCFT})\label{LPS}
For every $\eta>0$, there exist $\delta>0$ and $\kappa\ge 1$ such that at every $v\in \text{Reg}(\eta)$, the foliations $W_g^u, W_g^{cs}$ as well as  $W_g^s$, $W_g^{cu}$ have LPS with constant $\kappa$ in a $\delta$-neighborhood of $v$.
\end{lemma}

\subsubsection{Decomposition and horospherical translations for frame flows}
We can define a decomposition of orbit segments on $FM$ in the following way. 
\begin{definition}\label{decom1}
Given $(x,t)\in FM\times \RR^+$, then $(\pi(x), t)\in  SM\times \RR^+$ has a $(\PPP, \GGG, \SSS)$-decomposition according to Definition \ref{decom}. We define  $(\tilde \PPP, \tilde\GGG, \tilde\SSS)$-decomposition of $(x,t)$ as the unique lift of  $(\PPP, \GGG, \SSS)$-decomposition of $(\pi(x), t)$.
\end{definition}

In negative curvature case, the (strong) stable/unstable foliations of the frame flow are defined via the (strong) stable/unstable foliations of the geodesic flow by Brin \cite{Brin80}. More precisely, if $x\in FM$, then the stable leaf $W_F^s(x)$ sits above the stable leaf $W_g^s(\pi x)$ of the geodesic flow, and $y\in W_F^s(x)$ if $\pi(y)\in W_g^s(\pi x)$ and $d(F^ty, F^tx)\to 0$ as $t\to \infty$.
Given $v,w\in SM$ and $w\in W^s_g(v)$, the \emph{stable horospherical translation} $p^s(v,w): F_vM\to F_wM$ is defined as 
$$p^s(v,w)(x):=y\in \pi^{-1}(w)\cap W_F^s(x).$$
The lifting of $p^s$ to $FX$ is still called stable horospherical translation. Then the horospherical translation is an isometry between $F_vM$ and $F_wM$ and one can think of $p^s(v,w)(x)$ as the result of parallel translating $x$ along $\gamma_{v}(t)$ out to the boundary at infinity of $X$ and then back to $w$ along $\gamma_w(t)$. 

In nonpositive curvature, Constantine \cite{Con} defined stable horospherical translation via Brin's method when $w\in W^s_g(v)$ and $d(g^tv, g^tw)$ approaches to zero exponentially fast as $t\to \infty$.
\begin{lemma}(\cite[Proposition 2.1]{Con})\label{good}
Let $x\in FM$ and $v=\pi (x)\in SM$. Let $w\in W^s_g(v)$ (resp. $w\in W^u_g(v))$ and $d(g^tv, g^tw)$ approaches to zero exponentially fast as $t\to \infty$ (resp. $t\to -\infty$). Then there exists a unique $y\in \pi^{-1}(w)$ such
that the distance between $F^t(x)$ and $F^t(y)$ goes to zero as $t\to \infty$ (resp. $t\to -\infty$). We define horospherical translation $p^s(v,w): F_vM\to F_wM$ taking $x$ to $y$.
\end{lemma}
Note that there is a bunch of vectors on a same stable leaf on $SM$ but their forward orbits do not approach to zero exponentially fast. Nevertheless, we can define stable/unstable horospherical translation $p^s/p^u$ for a finite orbit segment $(v,t)\in \GGG(\eta)$.

\begin{lemma}\label{stablemanifold}
Given $\eta>0$, let $\delta=\delta(\eta)$ be as in \eqref{smalldelta} and $(v,T)\in \GGG(\eta)$.
Let $x\in FM$ and $v=\pi (x)\in SM$. If $w\in W^s_g(v,\delta)$, then there exists $y\in \pi^{-1}(w)$ such that for $0\le t\le T$,
$$d(F^ty, F^tx)\le Cd^s(v,w)e^{-\frac{\eta}{2}t}$$
for some universal $C>0$. In this way, we define the \emph{stable horospherical translation} $p^s(v,w)(x)=y$.

Analogously, if $w\in g^{-T}W_g^u(g^Tv,\delta)$, then there exists $z\in \pi^{-1}(w)$ for every $0\le t\le T$, 
$$d(F^tz, F^tx)\le Cd^u(F^Tv,F^Tw)e^{-\frac{\eta}{2}(T-t)}$$
for some universal $C>0$. We define the \emph{unstable horospherical translation} as $p^u(v,w)(x)=y$.
\end{lemma}
\begin{proof}
We only prove the statement for stable manifold and the one for unstable manifold is analogous. Let $w\in W^s_g(v,\delta)$. For any  $0\le t\le T$, we know by Lemma \ref{stablegeo} that $d^s(g^tv, g^tw)\le d^s(v,w)e^{-\frac{\eta}{2}t}$. We denote by $x_t\in F_{w}M$ such that $F^tx_t$ is the parallel translate of the frame $F^tx$ along the geodesic from $\pi (g^tv)$ to $\pi (g^tw)$. Since the parallel translations on frames are smooth actions, we have
$$d(F^tx_t, F^tx)\le D_1d(\pi (g^tv), \pi (g^tw))\le D_1d^s(g^tv, g^tw)\le D_1d^s(v,w)e^{-\frac{\eta}{2}t}$$
for some $D_1>1$ very close to $1$.
Denote $y=x_T$. Recall that the frame flow $F^t$ acts isometrically on the fiber $\pi^{-1}w$. Let $a>0$ be small enough and such that $T=na$ for some integer $n$. Then $\|dF^a\|\le C_1$ for some $C_1>1$ close to $1$. We have 
\begin{equation*}
    \begin{aligned}
d(x_t, x_{t+a})&= d(F^{t+a}x_{t}, F^{t+a}x)+d(F^{t+a}x, F^{t+a}x_{t+a})\\
&\le C_1d(F^tx_t, F^tx)+ d^s(v,w)e^{-\frac{\eta}{2}(t+a)}\\
&\le C_1D_1d^s(v,w)e^{-\frac{\eta}{2}t}+ D_1d^s(v,w)e^{-\frac{\eta}{2}(t+a)}\\
&\le 2d^s(v,w)e^{-\frac{\eta}{2}t}.
\end{aligned}
\end{equation*}
Let $t_j=ja$ and then $T=t_n$. As a consequence, for each $0\le i\le n$
\begin{equation*}
    \begin{aligned}
d(x_{t_i}, y)\le \sum_{j=i}^{n-1}d(x_{t_j},x_{t_{j+1}})\le 2d^s(v,w)e^{-\frac{\eta}{2}t_i}/(1-e^{-\eta a/2})=2C_2(a)d^s(v,w)e^{-\frac{\eta}{2}t_i}
\end{aligned}
\end{equation*}
where $C_2(a)$ is a constant depending on $a$.
For any $0\le t\le T$, suppose that $t_{i}\le t<t_{i+1}$. Then we have
\begin{equation*}
    \begin{aligned}
&d(F^{t}x, F^{t}y)\le C_1d(F^{t_i}x, F^{t_i}y)\\
\le &C_1d(F^{t_i}x, F^{t_i}x_{t_i})+C_1d(F^{t_i}x_{t_i}, F^{t_i}y)\\
\le &C_1D_1d^s(v,w)e^{-\frac{\eta}{2}t_i}+C_1d(x_{t_i}, y)\\
\le &C_1D_1d^s(v,w)e^{-\frac{\eta}{2}t_i}+2C_2(a)d^s(v,w)e^{-\frac{\eta}{2}t_i}\\
=&Cd^s(v,w)e^{-\frac{\eta}{2}t}
\end{aligned}
\end{equation*}
for some constant $C>1$ depending on $a$. We are done.
\end{proof}

Let $x\in FM$ and $v=\pi(x)$. If $(v,T)\in \GGG(\eta)$, we denote 
$$\hat W^s_F(x, T, \delta):=\{p^s(v,w)(x): w\in W^s_g(v,\delta)\}$$
and analogously
$$\hat W^u_F(F^Tx, T, \delta):=F^T\{p^u(v,w)(x): w\in g^{-T}W^u_g(g^Tv,\delta)\}$$
where $p^{s/u}$ are horospherecal translations defined in Lemma \ref{stablemanifold}.
By smoothness of the parallel translations on frames, we know that $\hat W^s_F(x, T, \delta)$ and $\hat W^u_F(F^Tx, T, \delta)$ are smooth embedded submanifolds of dimension $n-1$. Then we define center stable/unstable manifolds as
$$\hat W^{cs}_F(x, T, \delta):=\{z: \pi(z)=\pi(F^ty) \text{\ for some\ } y\in \hat W^{s}_F(x, T, \delta), t\in \RR\}$$
and analogously
$$\hat W^{cu}_F(F^Tx, T, \delta):=\{z: \pi(z)=\pi(F^ty) \text{\ for some\ } y\in \hat W^u_F(F^Tx, T, \delta), t\in \RR\}.$$
They are submanifolds of $FM$ of codimension $n-1$.

\subsubsection{Backward non-expansion in center unstable direction}
We start with a geometric characterization of the upper bound over the distance between a vector and its parallel transport over a null-homotopic smooth curve. The bound involves the integral of curvature tensor in the homotopy, and this is where (and in fact the only place) bunched curvature condition over $M$ is used. Details are stated in the following lemma. 
\begin{lemma} \label{lemmaholonomy}
    Let $S\subset M$ be a smoothly immersed disk with smooth boundary $\partial S$. Then there exists a constant $C_B'>1$ such that for all $p\in \partial S$ and $v\in S_pM$, we have
    $$
d(P_{\partial S}v,v)\leq \int_S C_B'|K(q)|dA(q)
    $$
    where $P_{\partial S}$ is the parallel transport over $\partial S$ and $K(q)$ is the sectional curvature of $T_qS$ for each $q\in S$.
\end{lemma}

Such a bound can be considered as a high-dimensional generalization of Gauss-Bonnet, and the idea is implicitly explained in the literature, see for instance \cite[page 247]{BM94}. For self-containment, we will give a brief proof as follows.
\begin{proof}[Proof of Lemma \ref{lemmaholonomy}]
    Fix any $p$ and $v$ as in the statement of the lemma, and parametrize $\partial S$ by a smooth curve $c:[0,1]\to M$ with $c(0)=c(1)=p$ with parameter $t$. By our assumption on $S$, we are able to build a smooth homotopy $C:[0,1]\times [0,1]\to M$ 
    with parameters $(s,t)$  satisfying
    $$
C(0,t)=p,\quad C(1,t)=c(t), \quad C(s,0)=C(s,1)=p.
    $$
    Let $v(t)\in S_{c(t)}M$ be the parallel transport of $v$ along $c[0,t]$ for each $t\in [0,1]$. By parallel transporting $v(t)$ in $s$-direction, we are able to construct a smooth section $v(s,t)\in S_{c(s,t)}M$ satisfying
    $$
v(1,t)=v(t), \quad \nabla_{\partial _s}v(s,t)=0 \text{ for all }(s,t)\in [0,1]\times [0,1].
    $$
    Fix any $w\in S_pM$, we will evaluate $d(P_{\partial S}v,v)$ by considering $\langle P_{\partial S}v-v, w \rangle$, with $w$ varying in $S_pM$. Similar to the construction of $v(s,t)$, we construct $w(s,t)\in S_{c(s,t)}M$ by 
    $$
w(0,t)=w, \quad \nabla_{\partial _s}w(s,t)=0 \text{ for all }(s,t)\in [0,1]\times [0,1].
    $$
    Then 
    $$
\begin{aligned}
    &\langle P_{\partial S}v-v, w \rangle\\
    =&\langle  v(0,1),w(0,1)  \rangle - \langle  v(0,0),w(0,0)  \rangle \\
    =&\int_{0}^1  \nabla_{\partial _t}(\langle v(0,t),w(0,t) \rangle) dt \\
    =&\int_{0}^1  \langle \nabla_{\partial _t} v(0,t),w(0,t) \rangle dt \\
    =&\int_{0}^1  \langle \nabla_{\partial _t} v(1,t),w(1,t) \rangle dt 
    -\int_{0}^1 \int_{0}^1 \nabla_{\partial _s}(\langle \nabla_{\partial _t} v(s,t),w(s,t) \rangle)dsdt \\
    =&-\int_{0}^1 \int_{0}^1 \nabla_{\partial _s}(\langle \nabla_{\partial _t} v(s,t),w(s,t) \rangle)dsdt \\
    =&-\int_{0}^1 \int_{0}^1 \langle \nabla_{\partial _s} \nabla_{\partial_t} v(s,t),w(s,t) \rangle dsdt \\
    =&\int_{0}^1 \int_{0}^1 \langle R(C^*\nabla_{\partial _t}, C^*\nabla_{\partial _s})v(s,t),w(s,t) \rangle dsdt
\end{aligned}
    $$
    where the third line follows from $\nabla_{\partial_t} w(0,t)=0$, the fifth line follows from $\nabla_{\partial _t} v(t)=0$, and the sixth line follows from $\nabla_{\partial _s}w(s,t)=0$. Since it is well-known that Riemannian curvature tensor can be represented by a finite sum of sectional curvature, the lemma follows from the curvature of $M$ being bunched and our arbitrary choice on $w\in S_pM$.
\end{proof} 

Fix a constant $\delta>0$ that is sufficiently small. The next lemma results from applying Lemma \ref{lemmaholonomy} to the `unstable parallelogram'.

\begin{lemma}  \label{lemmaunstableexpansion}
    There exists a constant $C^u>0$ such that the following holds: for all $T>0$, $v_1\in SM$, $v_2\in W_{g}^u(v_1,\delta)$, write $v_i':=g_{-T}(v_i)$ for $i\in \{1,2\}$. Let $r_{\alpha}(-T)$, $\alpha\in [0,1]$ be a smooth curve in $W^u_{g}(v_1',\delta)$ with $r_{0}(-T)=v_1'$, $r_{1}(-T)=v_2'$, and $l(g_T(r_{\alpha}(-T)))<2d^u(v_1,v_2)$. Let $y_1'\in FM$ be such that $\pi(y_1')=v_1'$, and $y_2'$ be the parallel transport of $y_1'$ along $r_{\alpha}(-T)$. Let $y_i$ be the parallel transport of $y_i'$ along $\gamma_{v'_i}$, $i\in \{1,2\}$. \footnote{Notice that $y_1=F_T(y_1')$, while $y_2$ is not necessarily equal to $F_T(y_2')$ since the first component of $y_2'$ might not be $v_2'$.} Then
\begin{equation} \label{eqrotationsmall}
    d(y_1,y_2)\leq C^ud^u(v_1,v_2).
\end{equation}
\end{lemma}
\begin{proof}
    Let $y_i:=(p_i,v_{0,i},v_{1,i},\cdots,v_{n-1,i})$ and $y_i':=(p'_i,v'_{0,i},v'_{1,i},\cdots,v'_{n-1,i})$ with $i\in \{1,2\}$, where $p_i=\pi(v_{0,i})$, $p'_i=\pi(v_{0,i}')$. Let $r_{\alpha}(t-T):=g_t(r_{-T}(\alpha))$ with $t\in [0,T]$, $y_1''$ be the parallel transport of $y_1$ along $r_{\alpha}(0),\alpha\in [0,1]$, $y_1'''$ be the parallel transport of $y_1''$ along the geodesic initiated from $-v_2$ for time $T$, and $\hat{y}_1$ be the parallel transport of $y_1'''$ along $r_{-\beta}(-T)$, $\beta\in [-1,0]$. Let $R:=\bigcup_{\alpha\in [0,1],t\in [-T,0]}r_{\alpha}(t)$ be the polygon in $TM$ bounded by $g_{[0,T]}v_1'$, $r_{[0,1]}(0)$, $g_{[0,T]}(-v_2)$, $r_{[1,0]}(-T)$ iteratively. It is clear that $\hat{y}_1=P_{\partial R}(y_1')$.
    
    By observing that $d(y_1,y_2)\leq d(y_1'',y_2)+l(r_{\alpha}(0))<d(y_1'',y_2)+2d^u(v_1,v_2)$ and $d(y_1'',y_2)=d(y_1''',y_2')=d(\hat{y}_1,y_1')$, it suffices to prove \eqref{eqrotationsmall} with LHS being replaced by $d(\hat{y}_1,y_1')$, which is
    \begin{equation} \label{eqrotationhatsmall}
        d(\hat{y}_1,y_1')\leq C^ud^u(v_1,v_2).
    \end{equation}
    Write $\hat{y}_1:=(p_1',\hat{v}_{0,1},\cdots,\hat{v}_{n-1,1})$, where $\hat{v}_{j,1}=P_{\partial R}v'_{j,1}$ for each $j$. To obtain \eqref{eqrotationhatsmall}, it suffices to show that there exists a constant $M^u>0$ such that for all $v\in S_{p_1'}(M)$, we have
    \begin{equation} \label{eqparalleltransportsmall}
        d(v,P_{\partial R}v)<M^u d^u(v_1,v_2).
    \end{equation}
    as \eqref{eqparalleltransportsmall} gives
    $$
\langle x, (P_{\partial R}-Id)y \rangle <nM^ud^u(v_1,v_2) \text{ for all }x,y\in F_{p_1'}M.
    $$

Now we turn to the proof of \eqref{eqparalleltransportsmall}. Since the curvature is uniformly bounded, smoothing $\partial R$ by smoothly connecting points on adjacent edges of $\partial R$ and pushing the points towards the common vertex (so in the limit case we have an infinitesimal triangle), we are able to apply Lemma \ref{lemmaholonomy} and obtain that 
$$
d(v, P_{\partial R}v)\leq C_B'\int_R |K|dA,
$$
where $K$ is the Gauss curvature of $R$. To prove \eqref{eqparalleltransportsmall}, it suffices to show the existence of some constant $C'>0$ such that
\begin{equation} \label{eqintegralcurvaturesmall}
    \int_{R}|K|dA\leq C'd^u(v_1,v_2).
\end{equation}

Write $K_{\alpha}(t):=K(r_{\alpha}(t))$, and $J_{\alpha}(t)$ as the unstable Jacobi field along $r_{\alpha}(t)$ for each $t\in [-T, 0]$. The Jacobi equation for $J_{\alpha}(t)$ along $r_{\alpha}$ on $R$ is given by
\begin{equation} \label{eqsurfaceJacobi}
    J''_{\alpha}(t)+K_{\alpha}(t)J_{\alpha}(t)=0.
\end{equation}
Writing $J_{\alpha}(t)=j_{\alpha}(t)E_{\alpha}(t)$, with $j_{\alpha}(t)>0$ and $||E_{\alpha}(t)||=1$. Then \eqref{eqsurfaceJacobi} turns into
\begin{equation} \label{eqsurfaceJacobinorm}
    j''_{\alpha}(t)+K_{\alpha}(t)j_{\alpha}(t)=0.
\end{equation}
Therefore, we have
$$
\begin{aligned}
    &\int_{R}|K|dA=\int_R -KdA 
    =\int_{-T}^0\int_0^1 -K_{\alpha}(t)j_{\alpha}(t)d\alpha dt\\
    =&\int_0^1\int_{-T}^0 j''_{\alpha}(t)dtd\alpha 
    =\int_0^1 j'_{\alpha}(0)-j'_{\alpha}(-T)d\alpha 
    \leq \int_0^1 j_{\alpha}(0)u_{\alpha}(0)d\alpha,
\end{aligned}
$$
where $u_{\alpha}(t):=\frac{j'_{\alpha}(t)}{j_{\alpha}(t)}$. By \cite[Lemma 2.9]{BCFT} and the fact of curvature being bounded, $u_{\alpha}(t)$ is uniformly bounded from above by a constant $U$, which is independent of the choice on $J_{\alpha}(t)$. Therefore, we may continue the above equation as 
$$
\int_{R}|K|dA\leq \int_0^1 j_{\alpha}(0)u_{\alpha}(0)d\alpha \leq U\int_0^1 j_{\alpha}(0)d\alpha=Ul(r_{\alpha}(0))<2Ud^u(v_1,v_2), 
$$
which concludes the proof of \eqref{eqintegralcurvaturesmall}, thus the proof of \eqref{eqrotationsmall}.
\end{proof}

The above lemma plays an essential role in the proof of the following result, which essentially indicates that $F$ can not expand too much in the center unstable direction when going backwards. 

\begin{proposition} \label{propnocontraction}
    There exists a constant $C^{cu}>1$ such that the following holds: for all $T>0$, $v_1^w\in SM$, $v_1^z\in W_{g}^{cu}(v_1^w,\delta)$, $w_1\in F_{v_1^w}M$, $z_1\in F_{v_1^z}M$, letting $w_1':=F^{-T}w_1$, $z_1':=F^{-T}z_1$, we have
    $$
d(z_1',w_1')\leq C^{cu}d(z_1,w_1).
    $$
\end{proposition}
\begin{proof}
    Let $v_2^w\in W^c_{g}(v_1^z,\delta)\cap W^u_{g}(v_1^w,\delta)$, and $v_1'^w,v_2'^w,v_1'^z$ be the $g_{-T}$-image of $v_1^w,v_2^w,v_1^z$ respectively. Let $r_{\alpha}(-T)\subset W_{g}^u(v_1'^w,\delta)$ be a smooth curve connecting $v_1'^w,v_2'^w$ with parameter $\alpha\in [0,1]$, such that $l(g_T(r_{\alpha}(-T)))<2d^u(v_1^w,v_2^w)$. Let $w_2'$ be the parallel transport of $w_1'$ along $r_{\alpha}(-T)$, and $w_2$ be the parallel transport of $w_2'$ along the geodesic corresponding to $(v_2'^w,T)$. We will operate the proof by discussing in cases.
    \begin{enumerate}
        \item $d(z_1,w_2)>2C^ud^u(v_1^w,v_2^w)$. In this case, by applying Lemma \ref{lemmaunstableexpansion} and noticing that $d(z_1,w_2)=d(z_1',w_2')$, we have
        $$
\begin{aligned}
    &d(z_1,w_1)\geq d(z_1,w_2)-d(w_1,w_2)\\
    \geq &d(z_1,w_2)-C^ud(v_1^w,v_2^w)
    \geq \frac{d(z_1,w_2)}{2}=\frac{d(z_1',w_2')}{2},
\end{aligned}
        $$
        which together with
        $$
        \begin{aligned}
&d(z_1',w_1')\leq d(z_1',w_2')+d(w_2',w_1')\leq d(z_1',w_2')+2d^u(v_1^w,v_2^w)\\
<&d(z_1',w_2')+\frac{d^u(z_1,w_2)}{C^u}=(1+\frac{1}{C^u})d(z_1',w_2')
       \end{aligned}
       $$
        implies that
\begin{equation} \label{eqcase1}
    d(z_1,w_1)>\frac{d(z_1',w_1')}{2(1+\frac{1}{C^u})}.
\end{equation}
        \item $d(z_1,w_2)<(2\kappa)^{-1}d^u(v_1^w,v_2^w)$, where $\kappa$ is from Lemma \ref{LPS}. In this case, we have
        $$
\begin{aligned}
    &d(z_1,w_1)
    \geq d(w_1,w_2)-d(z_1,w_2)\geq d(w_1,w_2)-(2\kappa)^{-1}d^u(v_1^w,v_2^w) \\
\geq &(\kappa)^{-1}d^u(v_1^w,v_2^w)-(2\kappa)^{-1}d^u(v_1^w,v_2^w)>(2\kappa)^{-1}d^u(v_1^w,v_2^w).
\end{aligned}
        $$
        Together with 
        $$
        \begin{aligned}
&d(z_1',w_1')\leq d(w_1',w_2')+d(w_2',z_1') \\
\leq &2d^u(v_1^w,v_2^w)+(2\kappa)^{-1}d^u(v_1^w,v_2^w)<3d^u(v_1^w,v_2^w),
       \end{aligned}
       $$
        we know
        \begin{equation} \label{eqcase2}
            d(z_1,w_1)>\frac{d(z_1',w_1')}{6\kappa}.
        \end{equation}
        \item $d(z_1,w_2)\in [(2\kappa)^{-1}d^u(v_1^w,v_2^w),2C^ud^u(v_1^w,v_2^w)]$. In this case, we have
\begin{equation} \label{eqcase3}
    \begin{aligned}
    &d(z_1,w_1)\geq d(\pi(z_1),\pi(w_1))\geq \kappa^{-1}d^{cu}(v_1^z,v_1^w)\\
    =&\kappa^{-1}(d^u(v_1^w,v_2^w)+d^c(v_2^w,v_1^z))>\kappa^{-1}(\frac{d^u(v_1^w,v_2^w)}{2}+\frac{d(z_1,w_2)}{4C^u})\\
    >&\kappa^{-1}(\frac{d(w_1',w_2')}{4}+\frac{d(z_1',w_2')}{4C^u})>\frac{d(z_1',w_1')}{4\kappa C^u}.
\end{aligned}
\end{equation}
    \end{enumerate}
    Combining \eqref{eqcase1}, \eqref{eqcase2}, \eqref{eqcase3}, we are done with the proof.
\end{proof}

\subsubsection{Weak specification}

\begin{lemma}\label{uniformtran}(Effective density)
Let $F^t:FM\to FM$ be the frame flow on a closed rank one manifold $M$ with bunched nonpositive curvature which is topologically transitive. Given any $\eta>0$, let $\delta=\delta(\eta)$ be as in \eqref{smalldelta}. Then there is $T\in \NN$ such that for every $(F^{-t_1}x,t_1), (y,t_2)\in \tilde \GGG(\eta)$, there is $0\leq t\leq T$ such that
$$F^t(\hat W_F^u(x,t_1, \delta))\cap \hat W_F^{cs}(y,t_2,\delta)\neq \emptyset.$$
\end{lemma}

\begin{proof}
By Lemma \ref{LPS}, the geodesic flow has local product structure in a closed $\delta$-neighborhood of every $v\in \text{Reg}(\eta)$.
Recall that the Smale bracket of two points $(v,w)\in  (\text{Reg}(\eta)\times B( \text{Reg}(\eta), \delta))\cup (B( \text{Reg}(\eta), \delta)\times \text{Reg}(\eta))$ is defined as
$$[v,w]_1:= W^{cs}_{g}(v, \kappa\delta)\cap W^u_{g}(w, \kappa\delta), [v,w]_2:= W^{cu}_{g}(v,\kappa\delta)\cap W^s_{g}(w,\kappa\delta)$$
where $\kappa$ is the constant in Lemma \ref{LPS}.
The function $$G(v,w):= \max \{d^{cs}([v,w]_1, v), d^u([v,w]_1, w), d^{cu}([v,w]_2, v), d^s([v,w]_2, w)\}$$
is continuous and vanishes when $v=w$.
Thus there is $\delta_1\ll \delta$ such that if $d(v,w)<2\kappa\delta_1$ then $G(v,w)<\delta/2$.
Moreover, there exists $0<\delta_2\ll \delta_1$ such that if $d(v,w)<\delta_2$ then $G(v,w)<\delta_1/2$.

Fix $x\in FM$ and denote $U=B(x,\delta_2)$. Define for any $t>0$
$$\gamma_t:=\sup\Big\{\gamma>0: \text{\ there exists\ } y\in FM \text{\ such that\ } B(y,\gamma)\cap \bigcup_{0\le s\le t}F^sU =\emptyset\Big\}.$$
We claim that $\gamma_t\to 0$ as $t\to \infty$. Indeed, otherwise there exist $y_t\in FM$ and $\gamma>0$ such that $B(y_t,\gamma)\cap \bigcup_{0\le s\le t}F^sU =\emptyset$ for all $t> 0$. Then any accumulation point $y$ of $y_t$ has $B(y,\gamma/2)\cap \bigcup_{s\ge 0} F^sU =\emptyset$, contradicting to the topological transitivity of $F^t$. This proves the claim, and it follows that there exists $T=T(x,\delta)$ such that for any $y\in FM$, $B(y,\delta_2)\cap \bigcup_{0\le s\le T}F^sU \neq\emptyset$.

We note that in the above argument $T=T(x,\delta)$ is depending on $x\in FM$. Noticing that $\bigcup_{0\le s\le T}F^sU$ and $B(y,\delta_2)$ are open sets, we know $x\mapsto T(x,\delta)$ is upper semicontinuous in $x$ and hence has a maximum value $T=T(\delta)$ over $FM$ which is exactly what we need.
 
Suppose that $(F^{-t_1}x,t_1), (y,t_2)\in \tilde \GGG(\eta)$. By the above argument,  there exist $z_1\in B(x,\delta_2)$ and $0\le T_0\le T=T(\delta)$ such that $F^{T_0}z_1\in B(y,\delta_2)$. Denote $z_2:=F^{T_0}z_1$. 

By the local product structure of the geodesic flow, let $a\in W_g^{s}(\pi(y),\delta)\cap W^{cu}_g(\pi (z_2),\delta)$.
Define $w_2:=p^s(\pi (y), a)(y)$. Then $w_2\in \hat W^{s}_F(y,t_2, \delta)$. Since $z_2\in B(y,\delta_2)$, we have
$$d(w_2,y)\le Cd^{cs}(a, \pi(y))\le C \delta_1/2$$
where $C$ is the constant in Lemma \ref{stablemanifold}. Then
$$d(w_2,z_2)\le d(w_2,y)+d(z_2,y)\le C\delta_1/2+\delta_2.$$
Denote $w_1:=F^{-T_0}w_2$ and $a_1=g^{-T_0}a$. Then $\pi(w_1)=a_1$. 
By Proposition \ref{propnocontraction}, we have that $d(z_1, w_1)\le C^{cu}d(z_2,w_2)$, which means that up to a scalar $C^{cu}$, $F^{-T_0}$ is non-expanding along $\pi^{-1}W_g^{cu}(\pi  z_2)$.

Since $g^{-t}$ is non-expanding along $W_g^{cu}$, we have 
$$d(a_1, \pi(z_1))\le d(a,\pi(z_2))< \delta_1/2.$$
Hence $d(\pi(x),a_1)\le d(\pi(x), \pi(z_1))+d(\pi(z_1), a_1) <\delta_1+\delta_1/2<2\delta_1.$
By the local product structure of the geodesic flow, there exist $b\in W^u_g(\pi (x),\delta)\cap W^{cs}_g(a_1,\delta)$ and $w=p^u(\pi(x), b)(x)$.
Then $\pi w=b$ and $w\in \hat W^u_F(x, t_1, \delta)$.
We have $d^u(b, \pi x)<2\kappa\delta_1$, and hence  $d(w, x)<2C^u\kappa\delta_1$ by Lemma \ref{lemmaunstableexpansion}. Thus 
$d(w,w_1)\le d(x,w)+d(x,z_1)+d(z_1,w_1)\le 2C^u\kappa\delta_1+\delta_2+C^{cu}(C\delta_1/2+\delta_2)$.

Since $w\in \pi^{-1}W_g^{cs}(\pi w_1)$, by a parallel argument as in the proof of Proposition \ref{propnocontraction}, we know that 
$d(F^{T_0}w, w_2)\le C^{cu}d(w, w_1)$. Thus we have
\begin{equation*}
    \begin{aligned}
&d(F^{T_0}w, y)\le d(F^{T_0}w, w_2)+d(w_2, y)\le C^{cu}d(w, w_1)+d(w_2,y)\\
\le &2C^{cu}C^u\kappa\delta_1+C^{cu}\delta_2+(C^{cu})^2(C\delta_1/2+\delta_2)+C\delta_1/2\ll \delta.
\end{aligned}
\end{equation*}
Thus $F^{T_0}w\in F^{T_0}(\hat W_F^u(x,t_1, \delta))\cap \hat W_F^{cs}(y,t_2,\delta)$ and we are done.
\end{proof}

\begin{theorem}\label{spe}
Let $F^t: FM \to FM$ be the frame flow on a closed rank one manifold $M$ with bunched nonpositive curvature which is topologically transitive. Given $\eta>0$, let $\delta=\delta(\eta)$ be as in \eqref{smalldelta} and $T=T((1-e^{-\frac{\eta}{2}})\delta)>0$ from Lemma \ref{uniformtran}. Then $\tilde{\mathcal{G}}(\eta)$ satisfies weak specification property with gap function $T$ at scale $\delta$. More precisely, for every $n\geq 2$ and 
$\{(x_i,t_i)\}_{i=1}^n\subset \tilde{\mathcal{G}}(\eta)$, there exist
\begin{enumerate}
    \item a sequence $\{\tau_i\}_{i=1}^{n-1}$ with $\tau_i\leq T$ for each $i$,
    \item a point $y\in X$, for which we denoted by $\text{Spec}^{n,\delta}_{\{\tau_i\}}(\{(x_i,t_i)\})$
\end{enumerate}
such that
$$
d_{t_i}(F^{\sum_{j=1}^{i-1}(t_j+\tau_j)}(y),x_i)<\delta.
$$
\end{theorem}

\begin{proof} 
Denote $\lambda=e^{\frac{\eta}{2}}$. We show inductively that for each $1\le k\le n$, there exist $y_k \in FM$ and $\tau_1,\tau_2, \dots, \tau_{k-1}$, such that for any $i=1,2,\cdots, k$ and $0\le t\le t_i$,
\begin{equation}\label{e:series}
d (F^tx_i,F^{t+\sum_{j=1}^{i-1}(t_j+\tau_j)}y_k)<(1+\frac{1}{\lambda}+\dots+\frac{1}{\lambda^{k-i}})(1-\frac{1}{\lambda})\delta.
\end{equation}
Setting $y=y_n$ will finish the proof of the theorem.

For $k=1$, let $y_1=x_1$, and the conclusion is obvious. Assume that the above claim is true for $1\le k\le n$. Then consider the case for $k+1$. We apply Lemma \ref{uniformtran} by replacing $x$ by $F^{\sum_{i=1}^{k-1}(t_i+\tau_i)}y_k$, $y$ by $x_{k+1}$, and $\delta$ by $(1-\frac{1}{\lambda})\delta$. Then we obtain $\tau_k\in [0, T((1-\frac{1}{\lambda})\delta)]$ such that
$$F^{-\sum_{i=1}^{k}(t_i+\tau_i)}\left(F^{\tau_k}(W_F^u(F^{\sum_{i=1}^{k-1}(t_i+\tau_i)+t_k}y_k,(1-\frac{1}{\lambda})\delta)\cap(\hat W_F^{cs}(x_{k+1},(1-\frac{1}{\lambda})\delta)))\right)\neq \emptyset.$$
Let $y_{k+1}$ be any point in the above intersection.

Considering $$F^{\sum_{i=1}^{k}(t_i+\tau_i)}(y_{k+1})\in \hat W_F^{cs}(x_{k+1},(1-\frac{1}{\lambda})\delta),$$ 
by Lemma \ref{stablemanifold} we can get for any $0\le t\le t_{k+1}$,
\begin{equation}
 d(F^tx_{k+1},F^{t+{\sum_{i=1}^{k}(t_i+\tau_i)}}y_{k+1})<(1-\frac{1}{\lambda})\delta.
\end{equation}
Considering $$F^{t_k+{\sum_{i=1}^{k-1}(t_i+\tau_i)}}(y_{k+1})\in \hat W_F^u(F^{t_k+\sum_{i=1}^{k-1}(t_i+\tau_i)}y_k,(1-\frac{1}{\lambda})\delta),$$ 
we have for each $0\le i \le k$ and $0\le t\le t_i$,
\begin{equation}
\begin{split}
&d(F^tx_i,F^{t+\sum_{j=1}^{i-1}(t_j+\tau_j)}y_{k+1})\\
\le&d(F^tx_i,F^{t+\sum_{j=1}^{i-1}(t_j+\tau_j)}y_k)+d(F^{t+{\sum_{j=1}^{i-1}(t_j+\tau_j)}}y_k,F^{t+{\sum_{j=1}^{i-1}(t_j+\tau_j)}}y_{k+1})\\
<&(1+\frac{1}{\lambda}+\dots+\frac{1}{\lambda^{k-i}})(1-\frac{1}{\lambda})\delta+\frac{1}{\lambda^{k+1-i}}(1-\frac{1}{\lambda})\delta\\
=&(1+\frac{1}{\lambda}+\dots+\frac{1}{\lambda^{k-i}}+\frac{1}{\lambda^{k+1-i}})(1-\frac{1}{\lambda})\delta.
\end{split}
\end{equation}
This proves \eqref{e:series} and we are done.
\end{proof}

\subsubsection{Uniqueness of equilibrium states}
Let $M$ be a closed, oriented, nonpositively curved $n$-manifold, with bunched curvature. Suppose that the frame flow $F^t:FM\to FM$ is topologically transitive, and $\tilde\varphi: FM\to \mathbb{R}$ is a H\"{o}lder continuous potential that is constant on the fibers of the bundle $\pi: FM\to SM$. Recall that $\tilde\varphi=\varphi\circ \pi$. The following is a version of Bowen's inequality \cite[Theorem 17]{Bow71} in our setting, for pressure of orbit segments.

\begin{lemma}\label{liftpressure}
For $\tilde{\mathcal{C}}\subset FM\times \RR^+$, denote $\mathcal{C}:=\{(\pi x, t): (x,t)\in \tilde{\mathcal{C}}\}\subset SM\times \RR^+$. Then we have $$P(F^t, \tilde{\mathcal{C}},\tilde\varphi)=P(g^t,\mathcal{C},\varphi).$$
\end{lemma}
\begin{proof}
Let $\epsilon>0$ and $T>0$. For any $v\in SM$, let $E_v$ be a $(T,\epsilon)$-spanning set of $\pi^{-1}v$ such that
$$
\Lambda_{sp,T}^{F,d}(\pi^{-1}v\times \RR^+,\tilde\varphi,\epsilon)=\sum_{x\in E_v}e^{\tilde{\Phi}_{0}(x,T)}.
$$
Recall that the frame flow $F^t$ acts isometrically and $\varphi$ is constant on the fiber $\pi^{-1}v$. So the choice of $E_v$ is independent on $T$. Then we have
$$
\Lambda_{sp,T}^{F,d}(\pi^{-1}v\times \RR^+,\tilde\varphi,\epsilon)=e^{\tilde\Phi_{0}(v,T)}\# E_v.
$$

Set $U_v:=\cup_{x\in E_v}B_T(x,2\epsilon)$, which is an open neighborhood of the fiber $\pi^{-1}v$. Then there is a $\gamma>0$ such that $\pi^{-1}B(v,\gamma)\subset U_v$. Set $W_v:=B(v,\gamma)$. Let $\{W_{v_i}\}_{i=1}^r$ be an open cover of $SM$ and $\delta>0$ the Lebesgue number of this cover. Let $E_T$ be a $(T,\delta)$-spanning set of $\mathcal{C}_T$ such that 
$$
\Lambda_{sp,T}^{g,d}(\mathcal{C},\varphi,\delta)=\sum_{w\in E_T}e^{\Phi_{0}(w,T)}.
$$
For $w\in E_T$, pick $c(w)$ among $v_1, v_2, \cdots, v_r$ such that $B(w, \delta)\subset W_{c(w)}$. Then for $w\in E_T$ and $z\in E_{c(w)}$, set 
\[
V(w;z):=\{x\in \tilde{\mathcal{C}}_T: d_T(x,z)<2\epsilon\}.
\]
It follows that $\tilde{\mathcal{C}}_T=\cup_{w,z}V(w;z)$.

Let $S$ be a $(T,4\epsilon)$-separated set of $\tilde{\mathcal{C}}_T$ such that 
$$
\Lambda_{T}^{F,d}(\tilde{\mathcal{C}},\tilde\varphi,4\epsilon)=\sum_{u\in S}e^{\tilde\Phi_{0}(u,T)}.
$$
Then each set $V(w;z)$ can contain at most one element of $S$. Therefore,
\begin{equation}
\begin{aligned}
&\Lambda_{T}^{F,d}(\tilde{\mathcal{C}},\tilde\varphi,4\epsilon)=\sum_{u\in S}e^{\tilde\Phi_{0}(u,T)}\\
\le&\sum_{w\in E_T}\sum_{u\in E_{c(w)}}e^{\tilde{\Phi}_{2\epsilon}(u,T)}\\
\le&\sum_{w\in E_T}e^{\Phi_{2\epsilon}(w,T)}\# E_{c(w)}.
\end{aligned}
\end{equation}
Noticing that $\# E_v$ has a uniform upper bound, we see from above that
$$P(F^t, \tilde{\mathcal{C}},\tilde\varphi, 4\epsilon,0)\le P(g^t,\mathcal{C},\varphi, \delta, 2\epsilon).$$
Letting $\epsilon\to 0$ and hence $\delta\to 0$, we get
$$P(F^t, \tilde{\mathcal{C}},\tilde\varphi)\le P(g^t,\mathcal{C},\varphi).$$
The other direction is straightforward, so we are done.
\end{proof}

\begin{proof}[Proof of Theorem \ref{frameflow}]
Recall that $(SM,g^t)$ is assumed to satisfy the pressure gap condition $P(g^t,\text{Sing},\varphi)<P(g^t,\varphi)$. By taking $\text{inj}(M)$ to be the injectivity radius of $M$, for any $\epsilon\in (0,\text{inj}(M)/3)$, it follows from \cite[Lemma 5.3]{BCFT} that $P_{\exp}^{\perp}(g^t,\ph,\epsilon)<P(g^t,\ph)$, which by \cite[Proposition 3.7]{CT16} implies that
\begin{equation} \label{eqpressuregtscale}
    P(g^t,\varphi,\epsilon')=P(g^t,\varphi) \quad \text{ for all }\epsilon'\in (0,\epsilon/2].
\end{equation}
Since $d(x,y)\ge d(\pi x, \pi y)$ for any $x,y\in FM$, we see $P(F^t,\tilde\varphi,\eta)\geq P(g^t,\varphi,\eta)$ for all $\eta>0$, and it follows immediately from Lemma \ref{liftpressure} and \eqref{eqpressuregtscale} that
\begin{equation} \label{eqpressureFscale}
    P(F^t,\tilde\varphi,\epsilon')=P(F^t,\tilde\varphi) \quad \text{ for all }\epsilon'\in (0,\text{inj}(M)/6)
\end{equation}

Meanwhile, for all such $\epsilon'$, it is well known that the geodesic flow in nonpositive curvature is entropy expansive at scale $\epsilon'$, i.e., $h(g_t,\Gamma_{\epsilon'}(v))=0$ for any $v\in SM$ where 
$$\Gamma_{\epsilon'}(v):=\{w\in SM: d(g^tv, g^tw)<\epsilon', \forall t\in \RR\}.$$ 
By Lemma \ref{liftpressure}, $h(F^t,\Gamma_{\epsilon'}(x))=0$ for any $x\in FM$ where 
$$\Gamma_{\epsilon'}(x):=\{y\in FM: d(F^ty, F^tx)<\epsilon', \forall t\in \RR\}.$$ 
This implies that the frame flow is also entropy expansive at scale $\epsilon'$. In particular, it follows from Proposition \ref{propexpansivegood} that 
\begin{equation}
    h_{\mu}(F^s,\mathcal{A}_s)=h_{\mu}(F^s)
\end{equation}
for every $s>0$, any finite partition $\mathcal{A}_s$ for $FM$ with $\text{Diam}_s(\mathcal{A}_s)<\epsilon'$, and any $\mu\in \mathcal{M}_{F}^e(FM)$ being almost entropy expansive at scale $\epsilon'$.

Given $(x,t)\in FM\times \RR^+$, $(\pi(x), t)\in  SM\times \RR^+$ has a $(\PPP, \GGG, \SSS)$-decomposition according to Definition \ref{decom}. We define $(\tilde \PPP, \tilde\GGG, \tilde\SSS)$-decomposition of $(x,t)$ as the unique lift of $(\PPP, \GGG, \SSS)$-decomposition of $(\pi(x), t)$, see Definition \ref{decom1}. We have the following properties:
 
\begin{enumerate}
    \item $\tilde \GGG$ satisfies weak specification, by Theorem \ref{spe}.
    \item Notice that $\varphi$ has the Bowen property on $\GGG$ at some scale $\delta_0$, see \cite[Corollaries 7.5 and 7.8]{BCFT}. As $\tilde \varphi$ is constant on each fiber, $\tilde\varphi$ has the Bowen property on $\tilde \GGG$ at the same scale.
    \item By \cite[Proposition 5.2]{BCFT}, $P(g^t, [\PPP]\cup [\SSS])<P(g^t, \varphi)$. It follows from Lemma \ref{liftpressure} that $P(F^t, [\tilde \PPP]\cup [\tilde \SSS])<P(F^t, \tilde\varphi)$.
\end{enumerate}   
Therefore, by taking $\eps:=\min\{\text{inj}(M)/6,\delta_0\}$, all the conditions of Theorem \ref{thm:weprove} are satisfied. It follows that the frame flow $(FM, F^t, \tilde\varphi)$ has a unique equilibrium state.
\end{proof}

\ \
\\[-2mm]
\textbf{Acknowledgement.}
We would like to thank Daniel Thompson for reading a first version of the paper and several useful remarks. TW would also like to thank Yuhao Hu for helpful discussions. TW and WW are supported by National Key R\&D Program of China No. 2022YFA1007800, NSFC No. 12071474; NSFC No. 12401239, and STCSM No. 24ZR1437200. 



\end{document}